\documentclass{article}

\usepackage{amsmath}
\usepackage{amsthm}
\usepackage{amssymb}
\usepackage{amsfonts}
\usepackage{mathtools}

\usepackage[margin=1.1in, top=1in, bottom=1in]{geometry}

\DeclareMathAccent{\widertilde}{\mathord}{largesymbols}{"62}

\newcommand{\nc}{\newcommand}
\nc{\nt}{\newtheorem} 
\nt{theorem}{Theorem}[section]
\nt{corollary}[theorem]{Corollary} 
\nt{proposition}[theorem]{Proposition}
\nt{lemma}[theorem]{Lemma} 
\nt{definition}[theorem]{Definition}
\nt{example}[theorem]{Example} 
\nt{remark}[theorem]{Remark}
\nt{assumption}[theorem]{Assumption}
\nt{alg}[theorem]{Algorithm}
\nt{con}[theorem]{Conjecture} 

\nc{\R}{{\bf R}}
\nc{\cl}{\mbox{\rm cl}\,} 
\nc{\cls}{ \mbox{{\scriptsize {\rm cl}}}\,} 
\nc{\conv}{\mbox{\rm conv}\,} 
\nc{\cone}{\mbox{\rm coco}\,} 
\nc{\lin}{\mbox{\rm lin}\,} 
\nc{\rb}{\mbox{\rm rb}\,}
\nc{\ri}{\mbox{\rm ri}\,}
\nc{\inter}{\mbox{\rm int}\,}
\nc{\bd}{\mbox{\rm bd}\,}
\nc{\relbd}{\mbox{\rm rlbd}\,}
\nc{\epi}{\mbox{\rm epi}\,}
\nc{\gph}{\mbox{\rm gph}\,}
\nc{\rge}{\mbox{\rm rge}\,}
\nc{\rgel}{\mbox{\rm {\scriptsize rge}}\,}
\nc{\sepi}{\mbox{\rm {\scriptsize epi}}\,}
\nc{\sdom}{\mbox{\rm {\scriptsize dom}}\,}
\nc{\sgph}{\mbox{\rm {\scriptsize gph}}\,}
\nc{\dom}{\mbox{\rm dom}\,}
\nc{\detr}{\mbox{\rm det}\,}
\nc{\para}{\mbox{\rm par}\,}
\nc{\crit}{\mbox{\rm crit}\,}
\nc{\dist}{\mbox{\rm dist}\,}
\nc{\kernal}{\mbox{\rm ker}\,}

\newcommand{\limsupp}{\operatornamewithlimits{Limsup}}
\newcommand{\lp}{\operatornamewithlimits{limsup}}
\newcommand{\lf}{\operatornamewithlimits{liminf}}
\newcommand{\argmax}{\operatornamewithlimits{argmax}}

\title{Clarke subgradients for directionally Lipschitzian stratifiable functions}
\author{D. Drusvyatskiy\thanks{%
    School of Operations Research and Information Engineering,
    Cornell University,
    Ithaca, New York, USA;
    {\tt http://people.orie.cornell.edu/dd379/}.
    Work of Dmitriy Drusvyatskiy on this paper has been partially supported by the NDSEG grant from the Department of Defense and by the BSF Travel Grant for Young Scientists.
    }%
	\and
	A. D. Ioffe\thanks{
	Department of Mathematics, Technion-Israel Institute of Technology, Haifa, Israel 32000;	
	{\tt http://www.math.technion.ac.il/Site/people/process.php?id=672}.
	Work of A. D. Ioffe was supported in part by the US-Israel Binational Science Foundation Grant 2008261.
	}
	\and	
  A. S. Lewis\thanks{%
  School of Operations Research and Information Engineering,
  Cornell University,
  Ithaca, New York, USA;
  {\tt http://people.orie.cornell.edu/aslewis/}.
  Research supported in part by National Science Foundation Grant DMS-0806057 and by the US-Israel Binational Scientific Foundation Grant 2008261.
}}

\begin{document}
\maketitle

\begin{abstract}
Using a geometric argument, we show that under a reasonable continuity condition, the Clarke subdifferential of a semi-algebraic (or more generally stratifiable) directionally Lipschitzian function admits a simple form: the normal cone to the domain and limits of gradients generate the entire Clarke subdifferential. The characterization formula we obtain unifies various apparently disparate results that have appeared in the literature. Our techniques also yield a simplified proof that closed semialgebraic functions on $\R^n$ have a limiting subdifferential graph of uniform local dimension $n$. \end{abstract}
\normalsize

\section{Introduction.}
Variational analysis, a subject that has been vigorously developing for the past 40 years, has proven itself to be extremely effective at describing nonsmooth phenomenon. The Clarke subdifferential (or generalized gradient) and the limiting subdifferential of a function are the earliest and most widely used constructions of the subject.
A key distinction between these two notions is that, in contrast to the limiting subdifferential, the Clarke subdifferential is always convex. 
From a computational point of view, convexity of the Clarke subdifferential is a great virtue. To illustrate, by the classical Rademacher theorem, a {\em locally Lipschitz continuous} function $f$ on an open subset $U$ of $\R^n$ is differentiable almost everywhere on $U$, in the sense of Lebesgue measure. Clarke, in \cite{origin}, showed that for such functions, the Clarke subdifferential admits the simple presentation
\begin{equation}\label{eqn:formula}
\partial_c f(\bar{x})=\conv \{\lim_{i\to\infty} \nabla f(x_i): x_i\stackrel{\Omega}{\rightarrow} \bar{x}\}, \tag{LipR}
\end{equation}
where $\bar{x}$ is any point of $U$ and $\Omega$ is any full measure subset of $U$. Such a formula holds great computational promise since gradients are often cheap to compute. For example, utilizing (\ref{eqn:formula}), Burke, Lewis, and Overton developed an effective computational scheme for approximating the Clarke subdifferential by sampling gradients \cite{BLO}, and, motivated by this idea,  developed a robust optimization algorithm \cite{alg}. 

The authors of \cite{BLO} further extended Clarke's result to the class
of {\em finite-valued}, {\em continuous} functions $f\colon U\to \R$, defined on an open subset $U$ of $\R^n$, that are {\em absolutely continuous on lines}, and are {\em directionally Lipschitzian}; the latter means that the Clarke normal cone to the epigraph of $f$ is pointed. Under these assumptions on $f$, the authors derived the representation
\begin{equation}\label{eqn:lewis}
\partial_c f(\bar{x})=\bigcap_{\delta > 0} \cl \conv\Big(\nabla f\big(\Omega\cap B_\delta(\bar{x})\big)\Big), \tag{ACLR}
\end{equation}
where $B_\delta(\bar{x})$ is an open ball of radius $\delta$ around $\bar{x}$ and $\Omega$ is any full measure subset of $U$, and they extended their computational scheme to this more general setting.
One can easily see that this formula generalizes Clarke's result, since locally Lipschitz functions are absolutely continuous on lines, and for such functions (\ref{eqn:lewis}) reduces to (\ref{eqn:formula}). Pointedness of the Clarke normal cone is a common theoretical assumption. For instance, closed convex sets with nonempty interior have this property. Some results related to (\ref{eqn:lewis}) appear in \cite{CMW}.

In optimization theory, one is often interested in {\em extended real-valued} functions (functions that are allowed to take on the value $+\infty$), so as to model constraints, for instance. Results above are not applicable in such instances. An early predecessor of (\ref{eqn:formula}) and (\ref{eqn:lewis}) does rectify this problem, at least when convexity is present. Rockafellar \cite[Theorem 25.6]{rock} showed that for any closed {\em convex} function $f\colon\R^n\to\R\cup\{+\infty\}$, whose domain $\dom f$ has a nonempty interior, the convex subdifferential has the form
\begin{equation}\label{eqn:rock}
\partial f(\bar{x})=\conv \{\lim_{i\to\infty} \nabla f(x_i): x_i\to \bar{x}\}+N_{\sdom f}(\bar{x}),\tag{CoR}
\end{equation}
where $\bar{x}$ is any point in the domain of $f$ and $N_{\sdom f}(\bar{x})$ is the normal cone to the domain of $f$ at $\bar{x}$. 

Our goal is to provide an intuitive and geometric proof of a representation formula unifying (\ref{eqn:formula}), (\ref{eqn:lewis}), and (\ref{eqn:rock}). To do so, we will impose a certain structural assumption on the functions $f$ that we consider. Namely, we will assume that the domain of $f$ can be locally ``stratified'' into a finite collection of smooth manifolds, so that $f$ is smooth on each such manifold. Many functions of practical importance in optimization and in nonsmooth analysis possess this property. All semi-algebraic functions (those functions whose graphs can be described as a union of finitely many sets, each defined by finitely many polynomial inequalities), and more generally, tame functions fall within this class \cite{tame_opt}. We will show (Theorem~\ref{thm:main}) that for a {\em directionally Lipschitzian}, {\em stratifiable} function $f\colon\R^n\to\R\cup\{+\infty\}$, that is continuous on its domain (for simplicity), the Clarke subdifferential admits the intuitive form
\begin{equation}\label{eqn:our}
\partial_c f(\bar{x})=\conv\{\lim_{i\to\infty} \nabla f(x_i):x_i \xrightarrow[]{\Omega}\bar{x}\} + \cone\{\lim_{\substack{i\to\infty\\ t_i \downarrow 0}}t_i\nabla f(x_i):x_i\xrightarrow[]{\Omega} \bar{x}\}+ N^{c}_{\sdom f}(\bar{x}),
\end{equation}
or equivalently,
$$\partial_c f(\bar{x})=\bigcap_{\delta >0} \cl\conv\Big(\nabla f\big(\Omega\cap B_\delta(\bar{x})\big)\Big)+ N^{c}_{\sdom f}(\bar{x}),$$ 
where $\Omega$ is any dense subset of $\dom f$ and $\cone$ denotes the convex conical hull. (In contrast to the aforementioned results, we do not require $\Omega$ to have full-measure).

This is significant both from theoretical and computational perspectives. Proofs of (\ref{eqn:formula}) and (\ref{eqn:lewis}) are based largely on Fubini's theorem and analysis of directional derivatives, and though the arguments are elegant, they do not shed light on the geometry driving such representations to hold. Similarly, Rockafellar's argument of (\ref{eqn:rock}) relies heavily on the well-oiled machinery of convex analysis. Consequently, a simple unified geometric argument is extremely desirable. From a practical point of view, representation (\ref{eqn:our}) decouples the behavior of the function from the geometry of the domain; consequently, when the domain is a simple set (polyhedral perhaps) and the behavior of the function on the interior of the domain is complex, our result provides a convenient method of calculating the Clarke subdifferential purely in terms of limits of gradients and the normal cone to the domain --- information that is often readily available. Furthermore, using (\ref{eqn:our}), the functions we consider in the current paper become amenable to the techniques developed in \cite{BLO}.
%
%
%
%

 Whereas (\ref{eqn:our}) deals with pointwise estimation of the Clarke subdifferential, our second result addresses the geometry of subdifferential graphs, as a whole. In particular, we consider the size of subdifferential graphs, a feature that may have important algorithmic applications. 
For instance, Robinson \cite{rob,newt} shows
computational promise for functions defined on $\R^n$ whose subdifferential
graphs are locally homeomorphic to an open subset of $\R^n$. 
Due to the results of Minty~\cite{minty} and Poliquin-Rockafellar~\cite{prox_reg}, Robinson's techniques are applicable for convex, and more generally, for ``prox-regular'' functions.
Trying to understand the size of subdifferential graphs in the absence of convexity (or monotonicity), the authors in \cite{small} were led to consider the semi-algebraic setting. The authors proved that the {\em limiting subdifferential graph} of a closed, proper, {\em semi-algebraic} function on $\R^n$ has {\em uniform local dimension} $n$. Applications to sensitivity analysis were also discussed. We show how the techniques developed in the current paper drastically simplify the proof of this striking fact. Remarkably, this dimensional uniformity does not hold for the Clarke subdifferential graph. 

The rest of the paper is organized as follows. In Section~\ref{sec:prelim}, we establish notation and recall some basic facts from variational analysis. In Section~\ref{sec:char}, we derive a characterization formula for the Clarke subdifferential of a directionally Lipschitzian, stratifiable function that possesses a certain continuity property on its domain, and in Section~\ref{sec:random} we relate our results to the gradient sampling framework.
In Section~\ref{sec:loc_dim}, we prove the theorem concerning the local dimension of semi-algebraic subdifferential graphs. We have designed this last section to be entirely independent from the previous ones (except for Section~\ref{sec:prelim}), since it does require a short foray into semi-algebraic geometry.

\section{Preliminary results.}\label{sec:prelim}
In this section, we summarize some of the fundamental tools used in variational analysis and nonsmooth optimization.
We refer the reader to the monographs Borwein-Zhu \cite{Borwein-Zhu}, Clarke-Ledyaev-Stern-Wolenski \cite{CLSW}, Mordukhovich \cite{Mord_1,Mord_2}, and Rockafellar-Wets \cite{VA}, for more details.  Unless otherwise stated, we follow the terminology and notation of \cite{VA}.

The functions we consider will take their values in the extended real line $\overline{\R}:=\R\cup\{-\infty,\infty\}$. We say that an extended-real-valued function is {\em proper} if it is never $-\infty$ and is not always $+\infty$.
For a function $f\colon\R^n\rightarrow\overline{\R}$, the {\em domain} of $f$ is $$\mbox{\rm dom}\, f:=\{x\in\R^n: f(x)<+\infty\},$$ and the {\em epigraph} of $f$ is $$\mbox{\rm epi}\, f:= \{(x,r)\in\R^n\times\R: r\geq f(x)\}.$$

Throughout this work, we will only use Euclidean norms. Hence for a point $x\in\R^n$, the symbol $|x|$ will denote the standard Euclidean norm of $x$. Unless we state otherwise, the topology on $\R^n$ that we consider is induced by this norm.
Given a set $Q\subset\R^n$, the notation $x_i\stackrel{Q}{\rightarrow} \bar{x}$ will mean that the sequence $x_i$ converges to $\bar{x}$ and all the points $x_i$ lie in $Q$.
We let ``$o(|x-\bar{x}|)$ for $x\in Q$'' be shorthand for a function that satisfies 
$\frac{o(|x-\bar{x}|)}{|x-\bar{x}|}\rightarrow 0$ whenever $x\stackrel{Q}{\rightarrow} \bar{x}$ with $x\neq\bar{x}$. 

A function $f\colon\R^n\to\overline{\R}$ is {\em locally Lipschitz continuous at} a point $\bar{x}$ {\em relative to} a set $Q\subset\R^n$ containing $\bar{x}$, if $f(\bar{x})$ is finite and there exists a a real number $\kappa\in [0,+\infty)$ with 
$$|f(x)-f(y)|\leq \kappa |x-y|, \textrm{ for all } x,y\in Q \textrm{ near } \bar{x}.$$
If there exists an open neighborhood $Q$ of $\bar{x}$ so that the above conditions hold, then we simply say that $f$ is locally Lipschitz continuous at $\bar{x}$.

A {\em set-valued mapping} $F$ from $\R^n$ to $\R^m$, denoted by $F\colon\R^n\rightrightarrows\R^m$, is a mapping from $\R^n$ to the power set of $\R^m$. Hence for each  point $x\in\R^n$, $F(x)$ is a subset of $\R^m$. For a set-valued mapping $F\colon\R^n\rightrightarrows\R^m$, the {\em domain} of $F$ is $$\mbox{\rm dom}\, F:=\{x\in\R^n:F(x)\neq\emptyset\},$$ and the {\em graph} of $F$ is $$\mbox{\rm gph}\, F:=\{(x,y)\in\R^n\times\R^m:y\in F(x)\}.$$
The {\em outer limit} of $F$ at $\bar{x}$ is 
$$\limsupp_{x\to\bar{x}} F(x):=\{v\in\R^m: \exists x_i\to\bar{x},~\exists v_i\to v \textrm{ with } v_i\in F(x_i) \}.$$
The mapping $F$ is {\em locally bounded} near $\bar{x}$ if the image $F(V)\subset\R^m$ is bounded, for some neighborhood $V$ of $\bar{x}$.
The following definition extends in two ways the classical notion of continuity to set-valued mappings.
\begin{definition}[Continuity]
{\rm Consider a set-valued mapping $F\colon\R^n\rightrightarrows\R^m$.
\begin{enumerate}
\item $F$ is {\em outer semicontinuous} at a point $\bar{x}\in\R^n$ if for any sequence of points $x_i\in\R^n$ converging to $\bar{x}$ and any sequence of vectors $v_i\in F(x_i)$ converging to $\bar{v}$, we must have $\bar{v}\in F(\bar{x})$.
\item $F$ is {\em inner semicontinuous} at $\bar{x}$ if for any sequence of points $x_i$ converging to $\bar{x}$ and any vector $\bar{v}\in F(\bar{x})$, there exist vectors $v_i\in F(x_i)$ converging to $\bar{v}$.
\end{enumerate}
If both properties hold, then we say that $F$ is {\em continuous} at $\bar{x}$. 
}
\end{definition}

We let $\inter Q$, $\cl Q$, $\conv Q$, and $\cone Q$, denote the interior, closure, convex hull, and convex conical hull of a set $Q$, respectively. A cone $Q$ is said to be {\em pointed} if it contains no lines. An open ball of radius $r$ around a point $\bar{x}\in\R^n$ will be denoted by $B_r(\bar{x})$. We let ${\bf B}$ and $\overline{{\bf B}}$ be the open and closed unit balls, respectively.
The following is a standard result on preservation of continuity of set-valued mappings under a pointwise convex conical hull operation. We provide a proof for completeness.

\begin{lemma}[Preservation of Continuity]\label{lem:outer}
Consider a set-valued mapping $F\colon\R^n\rightrightarrows\R^m$ that is outer-semicontinuous at a point $\bar{x}\in\R^n$. Suppose $F$ is locally bounded near $\bar{x}$ and $0\notin F(\bar{x})$.
Then the mapping $$x\mapsto G(x):=\cone F(x)$$ is outer-semicontinuous at $\bar{x}$, provided that $G(\bar{x})$ is pointed. 
\end{lemma}
\begin{proof}
Consider a sequence $(x_i,v_i)\to(\bar{x},\bar{v})$, with $v_i\in G(x_i)$ for each index $i$. By Carath\'{e}odory's theorem, we deduce 
$$v_i=\sum_{j=1}^{m}\lambda^i_j y^i_j,$$
for some multipliers $\lambda^i_j\geq 0$ and vectors $y^i_j\in F(x_i)$, where $j=1,\ldots, m$. 

Restricting to a subsequence, we may assume that there exist nonzero vectors $y_j\in F(\bar{x})$ satisfying  
$$y_j=\lim_{i\to\infty} y^i_j, \textrm{ for each index } j.$$ 

We claim that the sequence of multipliers $\lambda^i_j$ is bounded. Indeed suppose this is not the case and let $m_i:=\max_{j} \lambda^i_j$. Then up to a subsequence, there exist multipliers $\lambda_j$, not all zero, such that  $$\frac{\lambda^i_j}{m_i}\to\lambda_j \textrm{ as } i\to\infty, \textrm{ and }0=\sum_{j=1}^m \lambda_j y_j,$$ contradicting the fact that $G(\bar{x})$ is pointed. We conclude that the multipliers $\lambda^i_j$ are bounded. 

Then up to a subsequence, we have $$\bar{v}=\lim_{i\to\infty} \sum_{j=1}^{m}\lambda^i_j y^i_j= \sum_{j=1}^{m}\lambda_j y_j\in G(\bar{x}),$$ for some real numbers  $\lambda_j \geq 0$. We conclude that $G$ is outer-semicontinuous at $\bar{x}$.
\end{proof}

The distance of a point $x$ to a set $Q$ is  $$d_Q(x):=\inf_{y\in Q}|x-y|,$$ and the projection of $x$ onto $Q$ is $$P_Q(x):=\{y\in Q:|x-y|=d_Q(x)\}.$$

We now consider normal cones, which are fundamental objects in variational geometry.
\begin{definition}[Proximal normals]
{\rm
Consider a set $Q\subset\R^n$ and a point $\bar{x}\in Q$. The {\em proximal normal cone} to $Q$ at $\bar x$, denoted
$N^{P}_Q(\bar x)$, consists of all vectors $v \in \R^n$ such that $\bar{x}\in P_Q(\bar{x}+\frac{1}{r}v)$ for some $r>0$. 
}
\end{definition}

Geometrically, a vector $v\neq 0$ is a proximal normal to $Q$ at $\bar{x}$ precisely when there exists a ball touching $Q$ at $\bar{x}$ such that $v$ points from $\bar{x}$ towards the center of the ball. Furthermore, this condition amounts to 
$$\langle v,x-\bar{x} \rangle \leq O(|x-\bar{x}|^2) ~~\textrm{ as } x\to\bar{x} \textrm{ in } Q.$$

Relaxing the inequality above, one obtains the following notion.

\begin{definition}[Frech\'{e}t normals]
{\rm Consider a set $Q\subset\R^n$ and a point $\bar{x}\in Q$. The {\em Frech\'{e}t normal cone} to $Q$ at $\bar x$, denoted
$\hat N_Q(\bar x)$, consists of all vectors $v \in \R^n$ such that $$\langle v,x-\bar{x} \rangle \leq o(|x-\bar{x}|) ~~\textrm{ as }x\to\bar{x} \textrm{ in } Q.$$
}
\end{definition}
Note that both $N^P_Q(\bar{x})$ and $\hat{N}_Q(\bar{x})$ are convex cones, while $\hat{N}_Q(\bar{x})$ is also closed. Evidently, the set-valued mapping $x\mapsto \hat{N}_Q(x)$ is generally not outer-semicontinuous, and hence is not robust relative to perturbations in $x$. To correct for that, the following definition is introduced.
\begin{definition}[Limiting normals]
{\rm Consider a set $Q\subset\R^n$ and a point $\bar{x}\in Q$.  The {\em limiting normal cone} to $Q$ at $\bar{x}$, denoted $N_Q(\bar{x})$, consists of all vectors $v\in\R^n$ such that there are sequences $x_i\stackrel{Q}{\rightarrow} \bar{x}$ and $v_i\rightarrow v$ with $v_i\in\hat{N}_Q(x_i)$.}
\end{definition}

The limiting normal cone, as defined above, consists of limits of Frech\'{e}t normals. In fact, the same object arises if we only take limits of proximal normals \cite[Exercise 6.18]{VA}. Convexifying the limiting normal cone leads to the following definition.
\begin{definition}[Clarke normals]
{\rm
Consider a set $Q\subset\R^n$ and a point $\bar{x}\in Q$. The {\em Clarke normal cone} to $Q$ at $\bar{x}$ is 
$$N_Q^{c}(\bar{x}):=\cl\conv N_Q(\bar{x}).$$ }
\end{definition}

Given any set $Q\subset\R^n$ and a mapping $F\colon Q\to \widetilde{Q}$, where $\widetilde{Q}\subset\R^m$, we say that $F$ is ${\bf C}^p$-{\em smooth} $(p\geq 2)$ if for each point $\bar{x}\in Q$, there is a neighborhood $U$ of $\bar{x}$ and a ${\bf C}^p$ mapping $\hat{F}\colon \R^n\to\R^m$ that agrees with $F$ on $Q\cap U$. Henceforth, to simplify notation, the word smooth will mean ${\bf C}^2$-smooth.

\begin{definition}[Smooth Manifolds]
{\rm We say that a set $M$ in $\R^n$ is a ${\bf C}^2$-{\em submanifold} of dimension $r$ if for each point $\bar{x}\in M$, there is an open neighborhood $U$ around $\bar{x}$ and a mapping $F\colon \R^n\to\R^{n-r}$ that is ${\bf C}^2$-smooth with $\nabla F(\bar{x})$ of full rank, satisfying 
$M\cap U=\{x\in U: F(x)=0\}$.
}
\end{definition}
A good reference on smooth manifold theory is \cite{Lee}.

\begin{theorem}\cite[Example 6.8]{VA}\label{thm:clarke_man}
Consider a ${\bf C}^2$-manifold $M\subset\R^n$. Then at every point $x\in M$, the normal cone $N_M(x)$ is equal to the normal space to $M$ at $x$, in the sense of differential geometry.
\end{theorem}

In fact, the following stronger characterization holds.
\begin{theorem}[Prox-normal neighborhood]\label{thm:prox}
Consider a ${\bf C}^2$-manifold $M\subset\R^n$ and a point $\bar{x}\in M$. Then there exists an open neighborhood $U$ of $\bar{x}$, such that
\begin{enumerate}
\item the projection map $P_M$ is single-valued on $U$,
\item for any two points $x\in M\cap U$ and $v\in U$, the equivalence, $$v\in x+N_M(x) \Leftrightarrow x=P_M(v),$$ holds.
\end{enumerate}
\end{theorem}

Following the notation of \cite{Hare}, we call the set $U$ that is guaranteed to exist by Theorem~\ref{thm:prox},  a {\em prox-normal neighborhood} of $M$ at $\bar{x}$. For more details about Theorem~\ref{thm:prox}, see \cite[Exercise 13.38]{VA},  \cite[Proposition 1.9]{CLSW}. We should note that the theorem above holds for all ``prox-regular'' sets $M$ \cite{prox_reg}.

Armed with the aforementioned facts from variational geometry, we can study variational properties of functions via their subdifferential mappings.
\begin{definition}[Subdifferentials]
{\rm Consider a function $f\colon\R^n\rightarrow\overline{\R}$ and a point $\bar{x}\in\R^n$, with $f(\bar{x})$ finite. The {\em limiting subdifferential} of $f$ at $\bar{x}$ is defined by 
$$\partial f(\bar{x})= \{v\in\R^n: (v,-1)\in N_{\mbox{{\scriptsize {\rm epi}}}\, f}(\bar{x},f(\bar{x}))\}.$$
{\em Proximal}, {\em Frech\'{e}t}, and {\em Clarke subdifferentials} are defined analogously.
}
\end{definition}

\noindent For $\bar{x}$ such that $f(\bar{x})$ is not finite, we follow the convention that $\partial_{P} f(\bar{x})=\hat{\partial}f(\bar{x})=\partial f(\bar{x})=\partial_c f(\bar{x})=\emptyset$.

The subdifferentials defined above fail to capture the horizontal normals to the epigraph. Hence to obtain a more complete picture, we consider the following.
\begin{definition}[Horizon subdifferential]
{\rm For a function $f\colon\R^n\to\overline{\R}$ that is finite at a point $\bar{x}$, the {\em horizon subdifferential} is given by
$$\partial^{\infty} f(\bar{x})= \{v\in\R^n: (v,0)\in N_{\mbox{{\scriptsize {\rm epi}}}\, f}(\bar{x},f(\bar{x}))\}.$$}
\end{definition}

For a set $Q\subset\R^n$, we define $\delta_Q\colon\R^n\to\overline{\R}$ to be a function that is $0$ on $Q$ and $+\infty$ elsewhere. We call $\delta_Q$ the {\em indicator function} of $Q$.
Then for a point $\bar{x}$, we have $N_Q(\bar{x})=\partial \delta_Q(\bar{x})$, with analogous statements holding for the other subdifferentials.

Often, we will work with discontinuous functions $f\colon\R^n\to\overline{\R}$. For such functions, it is useful to consider $f${\em-attentive} convergence of a sequence $x_i$ to a point $\bar{x}$, denoted $x_i \xrightarrow[f]{} \bar{x}$. In this notation we have 
$$x_i \xrightarrow[f]{} \bar{x} \quad\Longleftrightarrow\quad x_i\to\bar{x} \textrm{ and } f(x_i)\to f(\bar{x}).$$ If in addition we have a set $Q\subset\R^n$, then $x_i \xrightarrow[f]{Q} \bar{x}$ will mean that $x_i$ converges $f$-attentively to $\bar{x}$ and the points $x_i$ all lie in $Q$. 
It is immediate that the mappings $\partial f$ and $\partial^{\infty} f$ are outer-semicontinuous with respect to $f$-attentive convergence $x \xrightarrow[f]{} \bar{x}$. 

%


Consider a function $f\colon\R^n\to\overline{\R}$ that is locally lower semi-continuous at a point $\bar{x}$, with $f(\bar{x})$ finite. Then $f$ is locally Lipschitz continuous around $\bar{x}$ if and only if the horizon subdifferential is trivial, that is  the condition $\partial^{\infty} f(\bar{x})=\{0\}$ holds \cite[Theorem 9.13]{VA}. Weakening the latter condition to requiring $\partial^{\infty} f(\bar{x})$ to simply be pointed, we arrive at the following central notion \cite[Exercise 9.42]{VA}.

\begin{definition}[epi-Lipschitzian sets and directionally Lipschitzian functions]
\hspace{1 mm}
{\rm
\begin{enumerate}
\item A set $Q\subset\R^n$ is {\em epi-Lipschitzian} at one of its points $\bar{x}$ if $Q$ is locally closed at $\bar{x}$ and the normal cone $N_Q(\bar{x})$ is pointed.  
\item A function $f\colon\R^n\to\overline{\R}$, that is finite at $\bar{x}$, is {\em directionally Lipschitzian} at $\bar{x}$ if $f$ is locally lower-semicontinuous at $\bar{x}$ and the cone $\partial^{\infty} f(\bar{x})$ is pointed.
\end{enumerate}}
\end{definition}	

Rockafellar \cite[Section 4]{Clarke_roc} proved that an epi-Lipschitzian set in $\R^n$, up to a rotation, locally coincides with an epigraph of a Lipschitz continuous function defined on $\R^{n-1}$. We should further note that the Clarke normal cone mapping of an epi-Lipschitzian set is outer-semicontinuous \cite[Proposition 6.8]{CLSW}.

It is easy to see that a function $f\colon\R^n\to\overline{\R}$ is directionally Lipschitzian at $\bar{x}$ if and only if the epigraph $\epi f$ is epi-Lipschitzian at $(\bar{x},f(\bar{x}))$. 
Furthermore, for a set $Q$ that is locally closed at $\bar{x}$, the limiting normal cone $N_Q(\bar{x})$ is pointed if and only if the Clarke normal cone $N^c_Q(\bar{x})$ is pointed \cite[Exercise 9.42]{VA}.

Consider the two functions 
\begin{displaymath}
   f_1(x)=x ~~\textrm{ and }~~ f_2(x)= \left\{
     \begin{array}{ll}
       x & \textrm{if } x \leq 0\\
       x+1 & \textrm{if } x > 0
     \end{array}
   \right.
\end{displaymath}
defined on the real line. Clearly both $f_1$ and $f_2$ are directionally Lipschitzian, and have the same derivatives at each point of differentiability. However $\partial_c f_1(0)\neq \partial_c f_2(0)$. Roughly speaking, this situation arises because some normal cones to the epigraph of a function $f$, namely at points $(x,r)$ with $r>f(x)$, may not correspond to any subdifferential. Consequently, if we have any hope of deriving a characterization of the Clarke subdifferential purely in terms of gradients and the normal cone to the domain, we must eliminate the situation above. Evidently, assumption of continuity of the function on the domain would do the trick. However, such an assumption would immediately eliminate many interesting convex functions from consideration. Rather than doing so, we identify a new condition, which arises naturally as a byproduct of our arguments. At the risk of sounding extravagant, we give this property a name.
\begin{definition}[Vertical continuity]
{\rm
We say that a function $f\colon\R^n\to\overline{\R}$ is {\em vertically continuous} at a point $\bar{x}\in\dom f$ if the equation
\begin{equation}
\limsupp_{\substack{ x_i\to\bar{x},~r\to f(\bar{x}) \\ r>f(\bar{x}) }   } N_{\sepi f}(x,r)= N_{\sdom f}(\bar{x})\times\{0\}, \label{eqn:strange}
\end{equation}
holds.
}
\end{definition}

To put this condition in perspective, we record the following observations.
\begin{proposition}[Properties of vertically continuous functions]\label{prop:norm}
Consider a proper function $f\colon\R^n\to\overline{\R}$ that is locally lower-semicontinuous at a point $\bar{x}$, with $f(\bar{x})$ finite. 
\begin{enumerate}
\item \label{it:0} Suppose that whenever a pair $(x,r)\in\epi f$, with $r > f(x)$, is near $(\bar{x},f(\bar{x}))$ we have
$$N_{\sepi f}(x,r)= N_{\sdom f}(x)\times\{0\}.$$ Then $f$ is vertically continuous at $\bar{x}$.
\item \label{it:1} Suppose that $\bar{x}$ lies in the interior of $\dom f$ and that $f$ is vertically continuous at $\bar{x}$. Then $f$ is continuous at $\bar{x}$, in the usual sense. 
\item \label{it:2} Suppose that $f$ is continuous on a neighborhood of $\bar{x}$, relative to the domain of $f$. Then $f$ is vertically continuous at all points of $\dom f$ near $\bar{x}$.
\item \label{it:3} If $f$ is convex, then $f$ is vertically continuous at every point $\bar{x}$ in $\dom f$.
\item \label{it:4} Suppose that $f$ is ``amenable'' at $\bar{x}$ in the sense of \cite{amen}; that is, $f$ is finite at $\bar{x}$ and there exists a neighborhood $V$ of $\bar{x}$ so that $f$ can be written as a composition $f=g\circ F$, for a ${\bf C}^1$ mapping $F\colon V\to\R^m$ and a proper, lower-semicontinuous, convex function $g\colon\R^m\to\overline{\R}$, so that the qualification condition 
\begin{equation}\label{eqn:qual}
N_{\sdom g}(F(\bar{x}))\cap \kernal \nabla F(\bar{x})^{*}=\{0\},
\end{equation}
is satisfied. Then $f$ is vertically continuous at $\bar{x}$.
\end{enumerate}
\end{proposition}
\begin{proof}
Claim \ref{it:0} is immediate from outer-semicontinuity of the normal cone map $N_{\sdom f}$.

To see \ref{it:1}, suppose that $f$ is vertically continuous at $\bar{x}\in\inter\dom f$. Since the normal cone to the domain of $f$ at $\bar{x}$ consists of the zero vector, we deduce $N_{\sepi f}(\bar{x},f(\bar{x})+\frac{1}{n})=\{0\}$, for all sufficiently large integers $n$. By \cite[Exercise 6.19]{VA}, we deduce that each such point $(\bar{x}, f(\bar{x})+\frac{1}{n})$ lies in the interior of the epigraph $\epi f$. Consequently for any sequence $x_i\to\bar{x}$, we have
$$\lp_{i\to\infty} f(x_i)\leq f(\bar{x})+\frac{1}{n},$$
for all large indices $n$. Letting $n$ tend to infinity, we deduce that $f$ is upper-semicontinuous at $\bar{x}$. The result follows.

To see \ref{it:2}, suppose that $f$ is continuous at $x\in\dom f$, relative to the domain of $f$. Then for any real $r>f(x)$ there exists an $\epsilon >0$ so that the epigraph $\epi f$ coincides with the product set, $\dom f\times [r-\epsilon,r+\epsilon]$, locally around $(x,r)$. In fact, this follows just from upper-semicontinuity of $f$ at $x$, relative to $\dom f$. We deduce $N_{\sepi f}(x,r)=N_{\sdom f}(x)\times \{0\}$. The result follows by \ref{it:0}.

To see \ref{it:3}, consider a pair $(\bar{x},r)$ with $r>f(\bar{x})$ and observe
\begin{align*}
(v,\alpha)\in N_{\sepi f}(\bar{x},r) &\Longleftrightarrow \langle (v,\alpha),(\bar{x},r)\rangle \geq \langle (v,\alpha),(x',r')\rangle ~~\textrm{ for all } (x',r')\in\epi f \\
&\Longleftrightarrow \alpha=0 \textrm{ and } v\in N_{\sdom f}(\bar{x}).
\end{align*}
Appealing to \ref{it:0}, we obtain the result.

The proof of \ref{it:4} follows from a standard nonsmooth chain rule. We outline the argument below. Without loss of generality, we can assume that the representation $f=g\circ F$ holds on all of $\R^n$. Observe $$\epi f=\{(x,r)\in\R^{n+1}: G(x,r)\in\epi g\},$$ for the mapping $G\colon\R^{n+1}\to\R^{m+1}$ defined by $G(x,r)=(F(x),r)$. We would like to use the chain rule appearing in \cite[Theorem 6.14]{VA} to compute the normal cone to $\epi f$. To this end, consider a pair $(x,r)\in \epi f$ and a vector $(y,\alpha)\in N_{\sepi g}(G(x,r))$. We have 
\begin{align*}
0=\nabla G(x,r)^{*} (y,\alpha) &\Longleftrightarrow \alpha= 0 \textrm{ and } \nabla F(x)^{*} y=0\\
&\Longleftrightarrow y\in N_{\sdom g}(F(x)) \textrm{ and } \nabla F(x)^{*} y=0\\
&\Longleftrightarrow y=0,
\end{align*}
where the last equivalence follows from the qualification condition (\ref{eqn:qual}). Applying the chain rule \cite[Theorem 6.14]{VA}, we deduce 
$$N_{\sepi f}(x,r)= \nabla G(x,r)^{*}N_{\sepi g}(G(x,r)),$$ for all pairs $(x,r)\in \epi f$. In particular, if we have $r > f(x)$, or equivalently $r> g(F(x))$, we deduce 
$$N_{\sepi f}(x,r)=\nabla F(x)^{*}N_{\sdom g}(F(x))\times \{0\}.$$
The right hand side coincides with $N_{\sdom f}(x)\times\{0\}$ by \cite[Theorem 3.3]{amen}. The result follows by appealing to \ref{it:0}.
\end{proof}

As can be seen from the proposition above, vertical continuity bridges the gap between continuity of the function on the interior of the domain and continuity on the whole domain, and hence the name. In summary, all convex and amenable functions have this property, as do functions that are continuous on their domains. An illustrative example is provided by the proper, lower semi-continuous, convex (directionally Lipschitzian) function $f$ on $\R^2$, defined by 
\begin{displaymath}
   f(x,y) = \left\{
     \begin{array}{ll}
       y^2/2x &\textrm{if } x >0\\
       0 &\textrm{if }x=0, y=0 \\
       \infty &\textrm{otherwise}
     \end{array}
   \right.
\end{displaymath} 
This function is discontinuous at the origin, despite being vertically continuous there.

In the sequel, we will need the following basic result.
\begin{proposition}\label{prop:help}
Consider a set $M\subset \R^n$ and a function $f\colon\R^n\to\overline{\R}$ that is finite-valued and smooth on $M$.
Then, at any point $\bar{x}\in M$, we have $$\hat{\partial} f(\bar{x})\subset \nabla g(\bar{x})+N_M(\bar{x}),$$ where $g\colon\R^n\to\R$ is any smooth function agreeing with $f$ on $M$ on a neighborhood of $\bar{x}$.
\end{proposition}
\begin{proof}
Define a function $h:\R^n\to\overline{\R}$ agreeing with $f$ on $M$ and equalling plus infinity elsewhere.
It is standard to check that the chain of inclusions,
$$\hat{\partial} f(\bar{x})\subset\hat{\partial} h(\bar{x})=\hat{\partial} (g(\cdot)+\delta_M(\cdot))(\bar{x})\subset\nabla g(\bar{x})+N_M(\bar{x}),$$ holds.
\end{proof}

\section{Characterization of the Clarke Subdifferential.}\label{sec:char}
Directionally Lipschitzian functions play an important role in optimization and are close relatives of locally Lipschitz functions~\cite{BLO, VA}. Indeed, Rockafellar  showed that epi-Lipschitzian sets, up to a change of coordinates, are epigraphs of Lipschitz functions \cite[Section 4]{Clarke_roc}. 

As was mentioned in the introduction, a key feature of Clarke's construction is that the Clarke subdifferential of a locally Lipschitz function $f$, on $\R^n$, can be described purely in terms of gradient information. It is then reasonable to hope that the same property holds for continuous directionally Lipschitzian functions, but this is too good to be true. Though such functions are differentiable almost everywhere \cite{cone_mon}, their gradients may fail to generate the entire Clarke subdifferential.
A simple example is furnished by the classical ternary Cantor function --- a nondecreasing, continuous, and therefore  directionally Lipschitzian function, with zero derivative at each point of differentiability. The Clarke subdifferential of this function does not identically consist of the zero vector \cite[Exercise 3.5.5]{Borwein-Zhu}, and consequently cannot be recovered from classical derivatives. This example notwithstanding, one does not expect the Cantor function to arise often in practice.

Nonsmoothness arises naturally in many applications, but not pathologically so. On the contrary, nonsmoothness is usually highly structured. Often such structure manifests itself through existence of a {\em stratification}. In the current work, we consider so-called {\em stratifiable} functions. Roughly speaking, domains of such function can be decomposed into smooth manifolds (called strata), which fit together in a ``regular'' way, and so that the function is smooth on each such stratum. In particular, this rich class of functions includes all semi-algebraic, and more generally, all o-minimally defined functions. See for example \cite{DM}. We now make this notion precise.
\begin{definition}[Locally finite stratifications]
{\rm Consider a set $Q$ in $\R^n$. A {\em locally finite stratification} of $Q$ is a partition of $Q$ into disjoint manifolds $M_i$ (called strata) satisfying 
\begin{itemize}
\item {\bf (frontier condition)} for each index $i$, the the closure of $M_i$ in $Q$ is the union of some $M_j$'s, and
\item {\bf (local finiteness)} each point $x\in Q$ has a neighborhood that intersects only finitely many strata.
\end{itemize}
We say that a set $Q\subset\R^n$ is {\em stratifiable} if it admits a locally finite stratification.
} 
\end{definition}

Observe that due to the frontier condition, a stratum $M_i$ intersects the closure of another stratum $M_j$ if and only if the inclusion $M_i\subset \cl M_j$ holds. Consequently, given a locally finite stratification of a set $Q$ into manifolds $\{M_i\}$, we can impose a natural partial order on the strata, namely 
$$M_i\preceq M_j \Leftrightarrow  M_i\subset \cl M_j.$$ 
A good example to keep in mind is the partition of a convex polyhedron into its open faces.

\begin{definition}[Stratifiable functions]
{\rm A function $f\colon\R^n\to\overline{\R}$ is {\em stratifiable} if there exists a locally finite stratification of $\dom f$ so that $f$ is smooth on each stratum.  }
\end{definition}

The following result nicely illustrates the geometric insight one obtains by working with stratifications explicitly.

\begin{proposition}[Dense differentiability]\label{prop:cl_int}
Consider a proper stratifiable function $f\colon\R^n\to\overline{\R}$ that is directionally Lipschitzian at all points of $\dom f$ near $\bar{x}$, and let $\Omega$ be any dense subset of $\dom f$. Then the set $\Omega\cap \dom\nabla f$ is dense in the domain of $f$, in the $f$-attentive sense, locally near $\bar{x}$. 
\end{proposition}
\begin{proof} Consider a locally finite stratification of $\dom f$ into manifolds $M_i$ so that $f$ is smooth on each stratum. Suppose for the sake of contradiction that there exists a point $x\in\dom f$ arbitrarily close to $\bar{x}$ and an $f$-attentive neighborhood $V=\{y\in\R^n:|y-x|<\epsilon, |f(y)-f(x)|<\delta\}$ so that $V\cap \Omega$ does not intersect any strata of dimension $n$. Shrinking $V$, we may assume that $V$ intersects only finitely many strata, say $\{M_j\}$ for $j\in J:=\{1,\ldots,k\}$, and that the inclusion $x\in\cl M_j$ holds for each index $j\in J$. Notice that since $f$ is continuous on each stratum, the set $V$ is a union of open subsets of the strata $M_j$ for $j\in J$.  

Now among the strata $M_j$ with $j\in J$, choose a stratum $M$ that is maximal with respect to the partial order $\preceq$.
Clearly, we have
$$M\cap \cl M_j=\emptyset, \textrm{ for each } j\in J \textrm{ with } M_j\neq M.$$
Now let $y$ be any point of $V\cap M$ and observe that there exists a neighborhood $Y$ of $y$ so that the functions $f$ and $f+\delta_{M}$ coincide on $Y\cap M$. We deduce that $\partial f(y)$ is a nontrivial affine subspace. Since $f$ is directionally Lipschitzian at all points in $\dom f$ near $\bar{x}$, and in particular at $y$, we have arrived at a contradiction. Thus $\Omega\cap\dom\nabla f$ is dense  (in the $f$-attentive sense) in the domain of $f$, locally near $\bar{x}$. 
\end{proof}

In this section, we will derive a characterization formula for the Clarke subdifferential of a stratifiable, vertically continuous, directionally Lipschitzian function $f\colon\R^n\to\overline{\R}$ . This formula will only depend on the gradients of $f$ and on the normal cone to the domain. It is important to note that the characterization formula, we obtain, is independent of any particular stratification of $\dom f$; one only needs to know that $f$ is stratifiable in order to apply our result. The argument we present is entirely constructive and is motivated by the following fact.
\begin{proposition}\label{prop:char}
Consider a closed, convex cone $Q\subset \R^n$, which is neither $\R^n$ nor a half-space. Then the equality,
$$Q=\cone (\bd Q).$$
holds.
\end{proposition}

Hence in light of Proposition~\ref{prop:char}, in order to obtain a representation formula for the Clarke subdifferential, it is sufficient to study the (relative) boundary structure of the Clarke normal cone. This is precisely the route we take.
\footnote{The idea to study the boundary structure of the Clarke normal cone in order to establish a convenient representation for the Clarke subdifferential is by no means new. For instance the same idea was used by Rockaffelar to establish the representation formula for the convex subdifferential \cite[Theorem 25.6]{rock}. While working on this paper, the authors became aware that the same strategy was also used to prove a representation formula for the subdifferential of finite-valued, continuous functions whose epigraph has positive reach 
\cite[Theorem 4.9]{CM}, \cite[Theorem 2]{CMW}. In particular, Proposition~\ref{prop:char} also appears as \cite[Proposition 3]{CMW}.
}

\begin{lemma}[Frech\'{e}t accessibility]\label{lem:access_set_main}
Consider a closed set $Q\subset\R^n$, a ${\bf C}^2$-manifold $M\subset Q$, and a point $\bar{x}\in M$. Recall that the inclusion $\hat{N}_Q(\bar{x})\subset N_M(\bar{x})$ holds. Suppose that a vector $\bar{v}\in \hat{N}_Q(\bar{x})$ lies in the boundary of $\hat{N}_Q(\bar{x})$, relative to the linear space $N_M(\bar{x})$. Then there exists a sequence $(x_i,v_i)\to(\bar{x},\bar{v})$, with  $v_i\in N^P_Q(x_i)$, and so that all points $x_i$ lie outside of $M$. 
\end{lemma}
\begin{proof} Choose a vector $\bar{w}\in N_M(\bar{x})$ in such a way so as to guarantee $$\bar{v}+t\bar{w}\notin \hat{N}_Q(\bar{x}), ~~\textrm{ for all } t>0.$$ 
Consider the vectors
\begin{equation}\label{eq:proj}
y(t):=\bar{x}+t(\bar{v}+t\bar{w}),
\end{equation}
and observe $y(t)\notin \bar{x}+\hat{N}_Q(\bar{x})$ for every $t>0$. Consider a selection of the projection operator, $$x(t)\in P_Q(y(t)).$$ Clearly, $y(t)\to \bar{x}$ and $x(t)\to \bar{x}$, as $t\to 0$. Observe 
\begin{align}
\frac{y(t)-x(t)}{t}&\in N^{P}_Q(x(t)), \notag\\
x(t)&\neq \bar{x},\label{eqn:neq}
\end{align}
for every $t$.

We claim that the points $x(t)$ all lie outside of $M$ for all sufficiently small $t>0$. Indeed, if this were not the case, then for sufficiently small $t$, the points $x(t)$ and $y(t)$ would lie in the prox-normal neighborhood of $M$ near $\bar{x}$, and we would deduce  
$$x(t)=P_M(y(t))=\bar{x},$$ contradicting (\ref{eqn:neq}).  

Thus all that is left is to show the convergence, 
$\frac{y(t)-x(t)}{t}\to\bar{v}$.
To this end, observe that from (\ref{eq:proj}), we have 
\begin{equation}\label{eqn:conv}
\frac{y(t)-\bar{x}}{t}\to \bar{v}.
\end{equation}
Hence it suffices to argue $\frac{x(t)-\bar{x}}{t}\to 0$.
By definition of $x(t)$, we have 
\begin{equation}\label{eqn:bounded}
|y(t)-\bar{x}|\geq |(y(t)-\bar{x})+(\bar{x}-x(t))|.
\end{equation}
Squaring and simplifying the inequality above, we obtain
\begin{equation}\label{eq:acc2}
\Big\langle \frac{y(t)-\bar{x}}{t}, \frac{x(t)-\bar{x}}{t}\Big\rangle \geq \frac{1}{2}\Big|\frac{x(t)-\bar{x}}{t}\Big|^2.
\end{equation} 

From (\ref{eqn:conv}) and (\ref{eqn:bounded}), we deduce that the vectors $\frac{x(t)-\bar{x}}{t}$ are bounded as $t\to 0$. Consider any limit point $\gamma \in\R^n$.
Taking the limit in (\ref{eq:acc2}), we obtain 
\begin{equation}\label{eqn:ineq}
\langle \bar{v}, \gamma\rangle \geq \frac{1}{2}|\gamma|^2.
\end{equation}
Since $\bar{v}$ is a Frech\'{e}t normal, we deduce 
$$\langle \bar{v},x(t)-\bar{x} \rangle \leq o(|x(t)-\bar{x}|).$$
It immediately follows that 
$$\langle \bar{v},\gamma\rangle \leq 0,$$
and in light of (\ref{eqn:ineq}),
we obtain $\gamma =0$. Hence 
$$\frac{y(t)-x(t)}{t}\to\bar{v},$$ as we claimed.
\end{proof}

\begin{remark}
{\rm
We note, in passing, that an analogue of Lemma~\ref{lem:access_set_main} (with an identical proof) holds when $M$ is simply ``prox-regular'', in the sense of \cite{prox_reg}, around $\bar{x}$. In particular, the lemma is valid when $M$ is a convex set. } 
\end{remark}


The combination of Lemma~\ref{lem:access_set_main} and Proposition~\ref{prop:char} yields dividends even in the simplest case when the manifold $M$ of Lemma~\ref{lem:access_set_main} is a singleton set. We should emphasize that in the following proposition, we do not even assume that the function in question is directionally Lipschitzian or stratifiable.

\begin{proposition}[Isolated singularity]\label{prop:iso}
Consider a continuous function $f\colon U\to\R$, defined on an open set $U\subset\R^n$. Suppose that $f$ is differentiable on $U\setminus \{\bar{x}\}$ for some point $\bar{x}\in U$, and that $\partial f(\bar{x})\neq \emptyset$. Then 
$$\partial_c f(\bar{x})= \cl \Big(\conv\{\lim_{i\to\infty} \nabla f(x_i):x_i\to\bar{x}\} + \cone\{\lim_{\substack{i\to\infty\\ t_i \downarrow 0}}t_i\nabla f(x_i):x_i\to \bar{x}\} \Big),$$ under the convention that $\conv \emptyset = \{0\}$.
\end{proposition}
\begin{proof}
Define the two sets 
\begin{equation*}
E:=\{\lim_{i\to\infty} \nabla f(x_i):x_i\to\bar{x}\},
~~~~H:= \{\lim_{\substack{i\to\infty\\ t_i \downarrow 0}}t_i\nabla f(x_i):x_i\to \bar{x}\},
\end{equation*}
and consider the epigraph $Q:=\epi f$ and the singleton set $M:=\{(\bar{x},f(\bar{x}))\}$. 

By Lemma~\ref{lem:access_set_main} and continuity of $f$, we have 
\begin{equation}\label{eq:try2}
\bd \hat{N}_Q(\bar{x},f(\bar{x}))\subset \cone\big(E\times\{-1\}\big)\cup \big(H\times\{0\}\big).
\end{equation}

{\sc Case 1.} Suppose $\hat{N}_Q(\bar{x},f(\bar{x}))$ is not equal to $\R^n\times [0,-\infty)$.
Then from Proposition~\ref{prop:char} and (\ref{eq:try2}), we deduce 
\begin{equation}\label{eqn:try1}
N^c_Q(\bar{x},f(\bar{x}))= \cl\cone\big(E\times\{-1\}\big)\cup \big(H\times\{0\}\big).
\end{equation}


From (\ref{eqn:try1}), we see that an inclusion $(v,-1)\in N^c_Q(\bar{x},f(\bar{x}))$ holds if and only if for every $\epsilon > 0$, there exist vectors $y_i\in E\cup H$, and real numbers $\lambda_i> 0$, for $1\leq i \leq  n+1$, satisfying
\begin{align*}
\Big|v&-\Big(\sum_{i:y_i\in E} \lambda_i y_i+ \sum_{i:y_i\in H} \lambda_i y_i\Big)\Big|<\epsilon,\\
1&=\sum_{i: y_i\in E} \lambda_i
\end{align*}
Thus $\partial_c f(\bar{x})=\cl (\conv E+\cone H)$, as we claimed.

{\sc Case 2.} Now suppose $\hat{N}_Q(\bar{x},f(\bar{x}))=\R^n\times [0,-\infty)$. Then from (\ref{eq:try2}), we deduce $H=\R^n$ and $\partial_c f(\bar{x})=\R^n=\conv E+\cone H$, under the convention $\conv \emptyset = \{0\}$. 
\end{proof}

As an illustration, consider the following simple example.
\begin{example}
{\rm Consider the function $f(x,y):=\sqrt[4]{x^4+y^2}$ on $\R^2$. Clearly $f$
is differentiable on $\R^2\setminus \{(0,0)\}$. The gradient has the form 
\[ \nabla f(x,y)=\frac{1}{(x^4+y^2)^{3/4}} \left( \begin{array}{c}
x^{3} \\
\frac{1}{2}y  \end{array} \right).\] 
From Proposition~\ref{prop:iso}, we obtain 
$$\partial_c f(\bar{x})= \cl\Big(\conv\{\lim_{i\to\infty} \nabla f(x_i):x_i\to\bar{x}\} + \cone\{\lim_{\substack{i\to\infty\\ t_i \downarrow 0}}t_i\nabla f(x_i):x_i\to \bar{x}\}\Big),$$
Observe that the vectors $\nabla f(x,0)$ are equal to $(\pm 1,0)$, and the vectors $2\sqrt{|y|} \nabla f(0,\pm y)$ are equal to $(0,\pm 1)$, whenever $x\neq 0\neq y$.
Thus we obtain
$$[-1,1]\times \{0\}\subset \conv\{\lim_{i\to\infty} \nabla f(x_i):x_i\to\bar{x}\}.$$
$$\{0\}\times \R\subset\cone\{\lim_{\substack{i\to\infty\\ t_i \downarrow 0}}t_i\nabla f(x_i):x_i\to \bar{x}\}.$$
Consequently, the inclusion $$[-1,1]\times\R \subset \partial_c f(0,0)$$ holds.
The absolute value of the first coordinate of $\nabla f(x,y)$ is always bounded by $1$, which implies the reverse inclusion above. Thus we have exact equality $\partial_c f(0,0)=[-1,1]\times\R$.
}
\end{example}

We record the following observation for ease of reference.
\begin{corollary}\label{cor:cont}
Consider a closed, convex cone $Q\subset\R^n$ with nonempty interior. Suppose that $\bd Q$ is contained in a proper linear subspace. Then $Q$ is either all of $\R^n$ or a half-space.
\end{corollary}
\begin{proof} Clearly if $Q$ were neither $\R^n$ or a half-space, then by Proposition~\ref{prop:char}, we would deduce that $Q=\cone (\bd Q)$ has empty interior, which is a contradiction.
\end{proof}

Armed with Proposition~\ref{prop:char} and Lemma~\ref{lem:access_set_main}, we can now prove the main result of this section with ease.

\begin{theorem}[Characterization]\label{thm:main}
Consider a proper, stratifiable function $f\colon\R^n\to\overline{\R}$ that is finite at $\bar{x}$. Suppose that $f$ is vertically continuous and directionally Lipschitzian at all points of $\dom f$ near $\bar{x}$. Then for any dense subset $\Omega\subset\dom f$, we have 
\begin{equation}\label{eqn:norm}
N^{c}_{\sepi f}(\bar{x},f(\bar{x}))=\cone \big\{\lim_{i\to\infty} \frac{(\nabla f(x_i),-1)}{\sqrt{1+|\nabla f(x_i)|}} : x_i\xrightarrow[f]{\Omega} 
\bar{x}\big\}+\big(N^{c}_{\sdom f}(\bar{x})\times \{0\}\big).    
\end{equation}
\noindent Consequently, the Clarke subdifferential admits the presentation
$$\partial_c f(\bar{x})=\conv\big\{\lim_{i\to\infty} \nabla f(x_i):x_i\xrightarrow[f]{\Omega}\bar{x}\big\} + \cone\big\{\lim_{\substack{i\to\infty\\ t_i \downarrow 0}}t_i\nabla f(x_i):x_i\xrightarrow[f]{\Omega} \bar{x}\big\}+ N^{c}_{\sdom f}(\bar{x}).$$
\end{theorem}
\begin{proof} 
We first prove (\ref{eqn:norm}). Observe that since $f$ is vertically continuous at $\bar{x}$, we have $N^{c}_{\sdom f}(\bar{x})\times \{0\}\subset N^{c}_{\sepi f}(\bar{x},f(\bar{x}))$, and hence the inclusion ``$\supset$'' holds.
Therefore we must establish the reverse inclusion. To this effect, intersecting the domain of $f$ with a small open ball around $\bar{x}$, we may assume that $f$ is directionally Lipschitzian and vertically continuous at each point $x\in \dom f$. For notational convenience, for a vector $v\in\R^{n}$, let $\overbracket[1pt]{v}:=\frac{(v,-1)}{\sqrt{1+|v|^2}}$. Define the set-valued mapping
\begin{equation}
F(x):= \cone \Big(\{\lim_{i\to\infty} \overbracket[1pt]{\nabla f(x_i)} : x_i\xrightarrow[f]{\Omega} x\}\cup (N_{\sdom f}(x)\cap {\bf B})\times\{0\}\Big),\label{eqn:set_val}
\end{equation}
By Proposition~\ref{prop:cl_int}, the set $\{\lim_{i\to\infty} \overbracket[1pt]{\nabla f(x_i)} : x_i\xrightarrow[f]{\Omega} x\}$ is nonempty. Furthermore, from the established inclusion ``$\supset$'', we see that $N_{\sdom f}(x)$ is pointed and hence the set
$\cone N_{\sdom f}(x)$ is closed for all $x\in\dom f$. Consequently, we deduce
\begin{equation*}
F(x)=\cone \{\lim_{i\to\infty} \overbracket[1pt]{\nabla f(x_i)}: x_i\xrightarrow[f]{\Omega} x\}+ \big(N^{c}_{\sdom f}(x)\times \{0\}\big),
\end{equation*}
Combining (\ref{eqn:set_val}) with Lemma~\ref{lem:outer}, we see that $F$ is outer-semicontinuous with respect to $f$-attentive convergence. 

Now consider a stratification of $\dom f$ into manifolds $\{M_i\}$ having the property that $f$ is smooth on each stratum $M_i$. Restricting the domain of $f$, we may assume that the stratification $\{M_i\}$ consists of only finitely many sets. We prove the theorem by induction on the dimension of the strata $M_i$ in which the point $\bar{x}$ lies. 

Clearly, the result holds for all strata of dimension $n$, since $f$ is smooth on such strata and $\Omega$ is dense in $\dom f$. As an inductive hypothesis, suppose that the claim holds for all strata that are strictly greater in the partial order $\preceq$ than a certain stratum $M$ and let $\bar{x}$ be an arbitrary point of $M$. 

Since $f$ is smooth on $M$, we deduce that $\gph f\big|_M$ is a smooth manifold. Then by Lemma~\ref{lem:access_set_main}, for every vector $0\neq v\in\rb \hat{N}_{\sepi f}(\bar{x},f(\bar{x}))$, there exists a sequence $(x_l,r_l,v_l)\to(\bar{x},f(\bar{x}),v)$, with $v_l\in \hat{N}_{\sepi f}(x_l,r_l)$ and $(x_l,r_l)\notin \gph f\big|_M$. Suppose that there exists a subsequence satisfying $r_l\neq f(x_l)$ for each index $l$. Then since $f$ is vertically continuous at $\bar{x}$, we obtain 
$$v=\lim_{l\to\infty} v_l\subset \limsupp_{l\to\infty}N_{\sepi f}(x_l,r_l)\subset N_{\sdom f}(\bar{x})\times \{0\}\subset F(\bar{x}).$$

On the other hand, if $r_l= f(x_l)$ for all large indices $i$, then restricting to a subsequence, we may assume that all the points $x_l$ lie in a a stratum $M'$ with $M'\succ M$. The inductive hypothesis and $f$-attentive outer-semicontinuity of $F$ yield the inclusion $v\in F(\bar{x})$.

Thus we have established the inclusion,
$$\rb \hat{N}_{\sepi f}(\bar{x},f(\bar{x}))\subset F(\bar{x}).$$
Since $\partial^{\infty} f (\bar{x})$ is pointed, we deduce that the cone $\hat{N}_{\sepi f}(\bar{x})$ is neither a linear subspace nor a half-subspace. Consequently by Lemma~\ref{lem:outer}, we deduce
\begin{equation}\label{eq:point}
\hat{N}_{\sepi f}(\bar{x},f(\bar{x}))=\cone\rb \hat{N}_{\sepi f}(\bar{x},f(\bar{x}))\subset F(\bar{x}).
\end{equation}
In fact, we have shown that (\ref{eq:point}) holds for all points $\bar{x}\in M$.

Finally consider a limiting normal $v\in N_{\sepi f}(\bar{x})$. Then there exists a sequence $(x_l,r_l,v_l)\to (\bar{x},f(\bar{x}),v)$, with $v_l\in  \hat{N}_{\sepi f}(x_l,r_l)$. It follows from (\ref{eq:point}), the inductive hypothesis, and $f$-attentive outer-semicontinuity of $F$ that the inclusion $v\in F(\bar{x})$ holds. Thus the induction is complete, as is the proof of (\ref{eqn:norm}).

To finish the proof of the theorem, define the two sets
\begin{equation}
E:=\{\lim_{i\to\infty} \nabla f(x_i):x_i\xrightarrow[f]{\Omega}\bar{x}\},~~~~
H:= \{\lim_{\substack{i\to\infty\\ t_i \downarrow 0}}t_i\nabla f(x_i):x_i\xrightarrow[f]{\Omega} \bar{x}\}.
\end{equation}
Observe 
$$\cone \{\lim_{i\to\infty} \overbracket[1pt]{\nabla f(x_i)} : x_i\xrightarrow[f]{\Omega} \bar{x}\}= \cone\big(E\times\{-1\}\big)\cup \big(H\times\{0\}\big).$$

Thus an inclusion $(v,-1)\in N^c_Q(\bar{x},f(\bar{x}))$ holds if and only if there exist vectors $y_i\in E\cup H$ and $y\in N^{c}_{\sdom f}(\bar{x})$, and real numbers $\lambda_i> 0$, for $1\leq i \leq  n+1$, satisfying
\begin{align*}
v&=\sum_{i:y_i\in E} \lambda_i y_i+ \sum_{i:y_i\in H} \lambda_i y_i +y,\\
1&=\sum_{i: y_i\in E} \lambda_i
\end{align*}
The result follows.
\end{proof}

Recovering representation (\ref{eqn:lewis}) of the introduction, in the setting of stratifiable functions, is now an easy task.
\begin{corollary}\label{cor:cite}
Consider a proper, stratifiable function $f\colon\R^n\to\overline{\R}$, that is finite at $\bar{x}$. Suppose that $f$ is directionally Lipschitzian at all points of $\dom f$ near $\bar{x}$, and is continuous near $\bar{x}$ relative to the domain of $f$. Then we have 
$$\partial_c f(\bar{x})=\bigcap_{\delta >0} \cl\conv\Big(\nabla f\big(\Omega\cap B_\delta(\bar{x})\big)\Big)+ N^{c}_{\sdom f}(\bar{x}),$$
where $\Omega$ is any dense subset of $\dom f$.
\end{corollary}
\begin{proof}
Since the cone $N^c_{\sepi f}(\bar{x},f(\bar{x}))$ is pointed, one can easily verify, much along the lines of Lemma~\ref{lem:outer}, that the equation
$$\bigcap_{\delta >0} 
\cl\cone\big\{\frac{(\nabla f(x),-1)}
{\sqrt{1+|\nabla f(x)|}}: x\in \Omega\cap B_{\delta}(\bar{x})\big\}=\cone \big\{\lim_{i\to\infty} \frac{(\nabla f(x_i),-1)}{\sqrt{1+|\nabla f(x_i)|}} : x_i\stackrel{\Omega}{\rightarrow} \bar{x}\big\},
$$
holds.
The result follows by an application of Theorem~\ref{thm:main}. We leave the details to the reader.  
\end{proof}



Our next goal is to recover the representation of the convex subdifferential (\ref{eqn:rock}) of the introduction in the stratifiable setting. In fact, we will consider the more general case of amenable functions.
Before proceeding, recall that a proper, lower-semicontinuous, convex function is directionally Lipschitzian at some point if and only if its domain has nonempty interior. A completely analogous situation occurs for amenable functions.

\begin{lemma}\label{lem:amen_dir}
Consider a function $f\colon\R^n\to\overline{\R}$ that is amenable at $\bar{x}$. 
Let $V$ be a neighborhood of $\bar{x}$ so that $f$ can be written as a composition $f=g\circ F$, for a ${\bf C}^1$ mapping $F\colon V\to\R^m$ and a proper, lower-semicontinuous, convex function $g\colon\R^m\to\overline{\R}$, so that the qualification condition 
\begin{equation*}
N_{\sdom g}(F(\bar{x}))\cap \kernal \nabla F(\bar{x})^{*}=\{0\},
\end{equation*}
is satisfied.
Then there exists a neighborhood $U$ of $\bar{x}$ so that 
\begin{enumerate}
\item \label{item:int} $F(U\cap\inter\dom f)\subset \inter\dom g$,
\item \label{item:bd} $U\cap F^{-1}(\inter\dom g) \subset \inter \dom f$.
\end{enumerate}
Furthermore $f$ is directionally Lipschitzian at $\bar{x}$ if and only if $\bar{x}$ lies in $\cl( \inter \dom f)$.
\end{lemma}
\begin{proof}
Let us first recall a few useful formulas. To this end, \cite[Theorem 3.3]{amen} shows that there exists a neighborhood $U$ of $\bar{x}$ so that 
for all points $x\in U\cap\dom f$, we have 
\begin{align}
\{0\}&=N_{\sdom g}(F(x))\cap \kernal \nabla F(x)^{*}, \label{align:prop} \\ 
\partial f(x)&=\nabla F(x)^{*}\partial g(F(x)),\\
N_{\sdom f}(x)&=\nabla F(x)^{*}N_{\sdom g}(F(x)). \label{align:prop2}
\end{align}
Furthermore a computation in the proof of Proposition~\ref{prop:norm} (item \ref{it:4}) shows that for any $x\in U\cap\dom f$ and any $r> f(x)$, we have 
\begin{equation}\label{eqn:epi_amin}
N_{\sepi f}(x,r)=\nabla F(x)^{*}N_{\sdom g}(F(x))\times \{0\}.
\end{equation}

Observe for any $x\in U\cap\inter\dom f$, we have  
$$0=N_{\sdom f}(x)=\nabla F(x)^{*}N_{\sdom g}(F(x)),$$
and consequently $N_{\sdom g}(F(x))=0$. We conclude $F(x)\in \inter\dom g$, thus establishing $\ref{item:int}$.


Now consider a point $x\in U\cap F^{-1}(\inter\dom g)$. 
Using $(\ref{eqn:epi_amin})$, we deduce $N_{\sepi f}(x,r)=0$ for any $r >f(x)$. Hence by \cite[Exercise 6.19]{VA}, we conclude $(x,r)\in\inter \epi f$ and consequently $x\in\inter \dom f$, thus establishing $\ref{item:bd}$.

By \cite[Exercise 10.25 (a)]{VA}, we have $\partial^{\infty} f(\bar{x})=N_{\sdom f}(\bar{x})$, and in light of (\ref{align:prop}) and (\ref{align:prop2}) one can readily verify that the cone $N_{\sdom f}(\bar{x})$ is pointed if and only if $N_{\sdom g}(F(\bar{x}))$ is pointed. 

Now suppose that $\bar{x}$ lies in $\cl(\inter \dom f)$. 
Then by $\ref{item:int}$ the domain $\dom g$ has nonempty interior and consequently $N_{\sdom g}(F(\bar{x}))$ is pointed, as is $N_{\sdom f}(\bar{x})$.

Conversely suppose that $f$ is directionally Lipschitzian at $\bar{x}$. Then $N_{\sdom g}(F(\bar{x}))$ is pointed, and consequently $\dom g$ has nonempty interior. Observe since $F$ is continuous, the set $F^{-1}(\inter \dom g)\subset\dom f$ is open. Hence it is sufficient to argue that this set contains $\bar{x}$ in its closure. 
Suppose this is not the case. Then there exists a neighborhood $U$ of $\bar{x}$ so that the image $F(U)$ does not intersect $\inter \dom g$. It follows that the range of the linearised mapping $w\mapsto F(\bar{x})+\nabla F(\bar{x})w$ can be separated from $\dom g$, thus contradicting $(\ref{align:prop})$. See \cite[Theorem 10.6]{VA} for a more detailed explanation of this latter assertion.
\end{proof}

We can now easily recover, in the stratifiable setting, representation (\ref{eqn:rock}) of the introduction. In fact, an entirely analogous formula holds more generally for amenable functions.
\begin{corollary}\label{cor:amen_rock}
Consider a proper stratifiable function $f\colon\R^n\to\overline{\R}$, that is amenable at a point $\bar{x}$, and so that $\bar{x}$ lies in the closure of the interior of $\dom f$. Let $\Omega$ be any dense subset of $\dom f$. Then the subdifferential admits the presentation
$$\partial f(\bar{x})=\conv\big\{\lim_{i\to\infty} \nabla f(x_i):x_i\stackrel{\Omega}{\rightarrow}\bar{x}\big\} + N_{\sdom f}(\bar{x}).$$
\end{corollary}
\begin{proof} By \cite[Exercise 10.25]{VA}, we have $\partial^{\infty} f(\bar{x})=N_{\sdom f}(\bar{x})$. 
Thus we have
$$\cone\big\{\lim_{\substack{i\to\infty\\ t_i \downarrow 0}}t_i\nabla f(x_i):x_i\xrightarrow[f]{\Omega} \bar{x}\big\}\subset N_{\sdom f}(\bar{x}).$$  

Observe $f$ is amenable, directionally Lipschitzian (Lemma~\ref{lem:amen_dir}), and vertically continuous (Proposition~\ref{prop:norm}) at each point of $\dom f$ near $\bar{x}$.
Applying Theorem~\ref{thm:main}, we deduce
$$\partial f(\bar{x})=\conv\big\{\lim_{i\to\infty} \nabla f(x_i):x_i\xrightarrow[f]{\Omega}\bar{x}\big\} + N_{\sdom f}(\bar{x}).$$ Noting that the subdifferential map $\partial f$ of an amenable function is outer-semicontinuous, the result follows.
\end{proof}

A natural question arises. Does the corollary above hold more generally without the stratifiability assumption? The answer turns out to be yes. This is immediate, in light of (\ref{eqn:rock}), for the subclass of {\em lower-${\bf C}^2$} functions (those functions that are locally representable as a difference of convex functions and convex quadratics).  A first attempt at a proof for general amenable functions might be to consider the representation $f=g\circ F$ and the chain rule
$$\partial f(x)=\nabla F(\bar{x})^{*}\partial g(F(\bar{x})).$$ One may then try to naively use Rockafellar's representation formula (\ref{eqn:rock}) for the convex subdifferential 
$$\partial g(F(\bar{x}))=\conv\{\lim_{i\to\infty} \nabla g(y_i): y_i\to F(\bar{x})\}$$
to deduce the result. However, we immediately run into trouble since $F$ may easily fail to be surjective onto a neighborhood of $F(\bar{x})$ in $\dom g$. Hence a different more sophisticated proof technique is required. For completeness, we present an argument below, which is a natural extension of the proof of \cite[Theorem 25.6]{rock}. It is furthermore instructive to emphasize how the stratifiability assumption allowed us in Corollary~\ref{cor:amen_rock} to bypass essentially all the technical details of the argument below. 

\begin{theorem}
Consider a function $f\colon\R^n\to\overline{\R}$ that is amenable at a point $\bar{x}$ lying in $\cl(\inter \dom f)$. Then the subdifferential admits the presentation
$$\partial f(\bar{x})=\conv\big\{\lim_{i\to\infty} \nabla f(x_i):x_i\stackrel{\Omega}{\rightarrow}\bar{x}\big\} + N_{\sdom f}(\bar{x}),$$
where $\Omega$ is any full-measure subset of $\dom f$.
\end{theorem}
\begin{proof}
Recall that $f$ is Clarke regular at $\bar{x}$, and therefore $\partial^{\infty} f(\bar{x})=N_{\sdom f}(\bar{x})$ is the recession cone of $\partial f(\bar{x})$. Combining this with the fact that the map $\partial f$ is outer-semicontinuous at $\bar{x}$, we immediately deduce the inclusion ``$\supset$''.   

We now argue the reverse inclusion. To this end, let $V$ be a neighborhood of $\bar{x}$ so that $f$ can be written as a composition $f=g\circ F$, for a ${\bf C}^1$ mapping $F\colon V\to\R^m$ and a proper, lower-semicontinuous, convex function $g\colon\R^m\to\overline{\R}$, so that the qualification condition 
\begin{equation*}
N_{\sdom g}(F(\bar{x}))\cap \kernal \nabla F(\bar{x})^{*}=\{0\},
\end{equation*}
is satisfied.
Since $f$ is directionally Lipschitzian at $\bar{x}$, the subdifferential $\partial f(x)$ is the sum of the convex hull of its extreme points and the recession cone $N_{\sdom f}(\bar{x})$. Furthermore every extreme point is a limit of exposed points. Thus
$$\partial f(\bar{x})= \conv (\cl E)+N_{\sdom f}(\bar{x}),$$
where $E$ is the set of all exposed point of $\partial f(\bar{x})$.

Hence to prove the needed inclusion, it suffices to argue the inclusion 
$$E\subset \conv\big\{\lim_{i\to\infty} \nabla f(x_i):x_i\stackrel{\Omega}{\rightarrow}\bar{x}\big\}.$$
Here, we should note that since $f$ is directionally Lipschitzian at $\bar{x}$, the set on the right hand side is closed.

To this end, let $\bar{v}$ be an arbitrary exposed point of $\partial f(\bar{x})$. By definition, there exists a vector $\bar{a}\in \R^n$ with $|\bar{a}|=1$ and satisfying
$$\langle \bar{a},\bar{v}\rangle > \langle \bar{a},v\rangle \textrm{ for all } v\in \partial f(\bar{x}) \textrm{ with } v\neq \bar{v}.$$
Since $N_{\sdom f}(\bar{x})$ is the recession cone of $\partial f(\bar{x})$, from above we deduce
$$\langle \bar{a}, z\rangle < 0 \textrm{ for all } 0\neq z\in N_{\sdom f}(\bar{x}),$$
and consequently
$$\langle \nabla F(\bar{x})\bar{a}, w\rangle < 0 \textrm{ for all } 0\neq w\in N_{\sdom g}(F(\bar{x})).$$
Consider the half-line $\{F(\bar{x})+t\nabla F(\bar{x})\bar{a}: t\geq 0\}$. We claim that this half-line cannot be separated from $\dom g$. Indeed, otherwise there would exist a nonzero vector $\bar{w}\in N_{\sdom g}(\bar{x})$ so that for all $t>0$ and all $x\in \dom g$ we have
$$\langle x,\bar{w}\rangle\leq \langle F(\bar{x})+t\nabla F(\bar{x})\bar{a},\bar{w}\rangle < \langle F(\bar{x}),\bar{w}\rangle,$$
which is a contradiction. Hence by \cite[Theorem 11.3]{rock}, this half-line must meet the interior of $\dom g$. By convexity then there exists a real number $\alpha >0$ satisfying $$\{F(\bar{x})+t\nabla F(\bar{x})\bar{a}: 0 < t\leq \alpha\}\subset \inter(\dom g).$$

Consequently the points $F(\bar{x}+t\bar{a})$ lie in $\inter \dom g$ for all sufficiently small $t >0$.
By Lemma~\ref{lem:amen_dir}, we deduce that there exists a real number $\beta >0$
so that 
$$\{\bar{x}+t\bar{a}: 0<t\leq\beta\}\subset \inter (\dom f).$$ 
Hence $f$ is Lipschitz continuous at each point $\bar{x}+t\bar{a}$ (for $0<t\leq\beta$), and so from (\ref{eqn:formula}) we obtain 
\begin{equation}\label{eq:form}
\partial f(\bar{x}+t\bar{a})=\conv\{\lim_{j\to\infty}\nabla f(x_j):x_j\stackrel{\Omega}{\rightarrow}\bar{x}+t\bar{a}\}.
\end{equation}

Now choose a sequence $t_i\to 0$ and observe that by \cite[Theorem 24.6]{rock}, for any $\epsilon>0$ we have 
\begin{align*}
\partial g(F(\bar{x}+t_i\bar{a}))&\subset \argmax_{v\in\partial g(F(\bar{x}))}\langle \nabla F(\bar{x})\bar{a}, v\rangle+\epsilon{\bf B},\\
&= \argmax_{v\in\partial g(F(\bar{x}))}\langle \bar{a}, \nabla F(\bar{x})^{*}v\rangle+\epsilon{\bf B},
\end{align*}
for all large $i$.
We deduce, 
$$\nabla F(\bar{x})^{*}\partial g(F(\bar{x}+t_i\bar{a}))\subset \argmax_{w\in\partial f(\bar{x})}\langle \bar{a}, w\rangle+\epsilon{\bf B}=\bar{v}+\epsilon{\bf B}.$$
Thus there exists a sequence $w_i\in \partial g(F(\bar{x}+t_i\bar{a}))$ with $\nabla F(\bar{x})^{*}w_i\to \bar{v}$. Consequently the vectors $\nabla F(\bar{x}+t_i\bar{a})^{*}w_i\in \partial f(\bar{x}+t_i\bar{a})$ converge to $\bar{v}$. The result now easily follows from (\ref{eq:form}) and the fact that $f$ is directionally Lipschitzian at $\bar{x}$.
\end{proof}


The following is a further illustration of the applicability of our results to a wide variety of situations.
\begin{corollary}
Consider a proper stratifiable function $f\colon\R^n\to\overline{\R}$ that is locally Lipschitz continuous at a point $\bar{x}$, relative to $\dom f$. Suppose furthermore that $\dom f$ is an epi-Lipschitzian set at $\bar{x}$. Then the formula
$$\partial f(\bar{x})=\conv\big\{\lim_{i\to\infty} \nabla f(x_i):x_i\stackrel{\Omega}{\rightarrow}\bar{x}\big\} + N^{c}_{\sdom f}(\bar{x}),$$
holds, where $\Omega$ is any dense subset of $\dom f$.
\end{corollary}
\begin{proof} Since $f$ is locally Lipschitz near $\bar{x}$ relative to $\dom f$, there exists a globally Lipschitz function $\tilde{f}\colon\R^n\to\R$, agreeing with $f$ on $\dom f$ near $\bar{x}$. Hence, we have
$$f(x)=\tilde{f}(x)+\delta_{\sdom f}(x), \textrm{ locally near } \bar{x}.$$
Combining this with \cite[Exercise 10.10]{VA}, we deduce
$$\partial^{\infty} f(x)\subset N_{\sdom f}(x), \textrm{ for } x \textrm{ near } \bar{x}.$$
We conclude that $f$ is directionally Lipschitzian  at all points of $\dom f$ near $\bar{x}$, and furthermore since the gradients of $\tilde{f}$ are bounded near $\bar{x}$ so are the gradients of $f$.
The result follows immediately by an application of Proposition~\ref{prop:norm} and Theorem~\ref{thm:main}.
\end{proof}

\section{Comments on random sampling}\label{sec:random}
Consider a proper, stratifiable function $f\colon\R^n\to\overline{\R}$, that is finite at $\bar{x}$, directionally Lipschitzian at all points of $\dom f$ near $\bar{x}$, and is continuous near $\bar{x}$ relative to the domain of $f$. Let us suppose for the time being that the normal cone $N_{\sdom f}(\bar{x})$ is known exactly.  Our interest lies in approximating the Clarke subdifferential $\partial_c f(\bar{x})$. In light of 
Corollary~\ref{cor:cite}, one can quickly modify the method considered in \cite{BLO}, in order to achieve this goal.  We now describe this procedure in more detail.

Fix a certain radius $\delta >0$ and a sample space $\Lambda=B_{\delta}(\bar{x})$, along with an associated probability measure that is absolutely continuous with respect to the Lebesgue measure $\mu$ on $\R^n$. We assume that the corresponding density $\theta$ is strictly positive almost everywhere on $\Lambda$. Observe that the set $\Lambda\cap \dom \nabla f$ has nonempty interior, and consequently has positive probability. We can consider a sequence of independent trials with outcomes $x_i\in\Lambda$ for $i=1,2,\ldots$, and form trial sets $$C_k=\conv\{\nabla f(x_i): x_i\in \Lambda\cap\dom \nabla f,~ 1\leq i\leq k\}.$$ One may then hope that the sets $D_k:=C_k+ N^{c}_{\sdom f}(\bar{x})$ approximate the Clarke subdifferential $\partial_c f(\bar{x})$ well, as $k$ tends to infinity. One can verify this rigorously via some modifications of arguments made in \cite{BLO}. The starting point is the following result. We omit the proof since it is identical to the argument in \cite[Theorem 2.1]{BLO}.

\begin{theorem}[Limiting Approximation]\label{thm:la}
Consider a function $f\colon \R^n \to \overline{\R}$ and a point $\bar{x}\in\dom f$. Suppose that $f$ is continuously differentiable on an open set dense in $\dom f$. Then for any sampling radius $\delta$, we have 
$$\cl \bigcup^{\infty}_{k=1} C_k = \cl\conv \nabla f(B_{\delta}(\bar{x})), \textrm{ almost surely}.$$ 
\end{theorem}
The following theorem establishes that $\cl\bigcup^{\infty}_{k=1} D_k$ is almost surely an outer approximation of $\partial_c f(\bar{x})$.
\begin{theorem}[Outer approximation]\label{theorem:out}
Consider a proper, stratifiable function $f\colon\R^n\to\overline{\R}$, that is finite at $\bar{x}$. Suppose that $f$ is directionally Lipschitzian at all points of $\dom f$ near $\bar{x}$, and is continuous near $\bar{x}$ relative to the domain of $f$. Then 
$$\partial_c f(\bar{x})\subset \cl\bigcup^{\infty}_{k=1} D_k, \textrm{ almost surely}.$$
\end{theorem}
\begin{proof} This is immediate from Theorem~\ref{thm:la} and Corollary~\ref{cor:cite}.
\end{proof}

On the other hand, the following lemma shows that $\cl \bigcup^{\infty}_{k=1} D_k$ is not too much bigger than $\partial_c f(\bar{x})$, as long as the radius is sufficiently small and we restrict ourselves to considering only bounded subsets of $\R^n$.
\begin{lemma}[Truncated Inner approximation]\label{lem:trunc:inner}
Consider a proper, stratifiable function $f\colon\R^n\to\overline{\R}$, that is finite at $\bar{x}$. Suppose that $f$ is directionally Lipschitzian at all points of $\dom f$ near $\bar{x}$, and is continuous near $\bar{x}$ relative to the domain of $f$. Then for any compact subset $\Gamma \subset\R^n$ and a real number $\epsilon >0$ we have, for any sufficiently small sampling radius,
$$\Gamma \cap \cl \bigcup^{\infty}_{k=1} D_k\subset \Gamma \cap \partial_c f(\bar{x})+\epsilon {\bf B}, \textrm{ almost surely}.$$
If, in addition, $f$ is Lipschitz continuous on its domain, then for any real number $\epsilon >0$ we have, for any sufficiently small sampling radius,
$$\cl \bigcup^{\infty}_{k=1} D_k\subset \partial_c f(\bar{x})+\epsilon {\bf B}, \textrm{ almost surely}.$$
\end{lemma}
\begin{proof}
Since $f$ is directionally Lipschitzian at $\bar{x}$, one can easily see check that the mapping $\partial_c f$ is outer-semicontinuous at $\bar{x}$. Consequently the truncated mapping $x\mapsto \Gamma\cap\partial_c f(x)$ is upper-semicontinuous at $\bar{x}$. We conclude that there exists a real number $\delta > 0$ such that $$\Gamma\cap \partial_c f(\bar{x}+\delta {\bf B})\subset \Gamma\cap \partial_c f(\bar{x})+ \epsilon {\bf B}.$$ Observe that $\cl \bigcup^{\infty}_{k=1} D_k$ is almost surely contained in $\Gamma\cap \cl \conv \partial_c f(\bar{x}+\delta {\bf B})$, and hence the result follows.

Now suppose that $f$ is Lipschitz continuous on its domain. Then, one can verify that that there exists a real number $\delta > 0$ satisfying 
$$\nabla f(B_{\delta}(\bar{x}))\subset \partial_c f(\bar{x})+\epsilon {\bf B}.$$ Since $\cl \bigcup^{\infty}_{k=1} D_k$ is almost surely contained $\cl\conv\big(\nabla f(B_{\delta}(\bar{x}))+N_{\sdom f}(\bar{x})\big)$, the result follows.
\end{proof}


In particular, the distance of a fixed vector $v$ to $\cl \bigcup^{\infty}_{k=1} D_k$ can be made arbitrarily close to the distance of $v$ to $\partial_c f(\bar{x})$.  
\begin{theorem}[Distance approximation]\label{thm:dist}
Consider a proper, stratifiable function $f\colon\R^n\to\overline{\R}$, that is finite at $\bar{x}$. Suppose that $f$ is directionally Lipschitzian at all points of $\dom f$ near $\bar{x}$, and is continuous near $\bar{x}$ relative to the domain of $f$. Then for any vector $v\in\R^n$ and a real $\epsilon > 0$ we have, for any sufficiently small sampling radius,
$$\dist(v,\cl \bigcup^{\infty}_{k=1} D_k)\geq \dist(v,\partial_c f(\bar{x})) -\epsilon, \textrm{ almost surely},$$
and consequently
$$\lim_{k\to\infty}\big|\dist(v,\cl D_k)- \dist(v,\partial_c f(\bar{x}))\big|<\epsilon, \textrm{ almost surely}.$$
\end{theorem}
\begin{proof} Let $\gamma:= \dist(v,\partial_c f(\bar{x}))$ and $\Gamma=\overline{B}_{\gamma}(v)$. Then by Lemma~\ref{lem:trunc:inner}, we have for any sufficiently small radius,  
$$\Gamma \cap \cl \bigcup^{\infty}_{k=1} D_k\subset \Gamma \cap \partial_c f(\bar{x})+\epsilon {\bf B}, \textrm{ almost surely}.$$ The result readily follows from the inclusion above and Theorem~\ref{theorem:out}.
\end{proof}

A natural test for optimality, using the sampling scheme, is to determine whether the inclusion $$0\in D_k,$$ holds. According to Theorem~\ref{theorem:out}, if $\bar{x}$
is a Clarke critical point, then $\dist(0,D_k)\to 0$ almost surely. On the other hand, if $\bar{x}$ is not Clarke-critical, then by Theorem~\ref{thm:dist}, for any sufficiently small radius, we have $\lim_{k\to\infty} \dist(0,\cl D_k) >0$, and hence the test $0\in D_k$ will not generate a false positive.

From a computational point of view, there is a certain difficulty we have not addressed. Suppose that for each radius $\delta$, the trial points $x_i$ are sampled with uniform distribution on the ball $B_{\delta}(\bar{x})$. One could worry that as the sampling radius decreases to zero, the proportion of points $x_i$ discarded (those that lie outside of the domain) to those that are in the domain could be arbitrarily large, with positive probability. This, for instance, could happen if the domain was the set $\{(x,y)\in{\R^2}:|y|\leq x^2, 0\leq x\}$. This pathology, however, does not occur in the directionally Lipschitzian setting. 

\begin{proposition}[Domain of a continuous directionally Lipschitzian function]
Consider a proper function $f\colon\R^n\to\overline{\R}$, that is finite at $\bar{x}$. Suppose that $f$ is directionally Lipschitzian at $\bar{x}$ and is continuous at $\bar{x}$, relative to $\dom f$. Then the domain of $f$ is epi-Lipschitzian at $\bar{x}$.
\end{proposition}
\begin{proof}
By continuity of $f$, the domain of $f$ is locally closed near $\bar{x}$. Now observe that for each real number $r$ satisfying $r> f(\bar{x})$, we have $N_{\sepi f}(\bar{x},r)=N_{\sdom f}(\bar{x})\times\{0\}$. Consequently the inclusion $$N_{\sdom f}(\bar{x})\times\{0\}\subset N_{\sepi f}(\bar{x},f(\bar{x})),$$ holds. We conclude that the normal cone $N_{\sdom f}(\bar{x})$ is pointed, and hence the domain of $f$ is an epi-Lipschitzian set at $\bar{x}$.
\end{proof}

The {\em lower Lebesgue density} of a set $Q\subset\R^n$ at a point $\bar{x}$ is
$$\mbox{\rm Dens}^{-}(Q,\bar{x}):=\lf_{\delta\to 0} \frac{\mu(Q\cap B_{\delta}(\bar{x}))}{\mu(B_{\delta}(\bar{x}))},$$ where we recall that $\mu$ denotes the Lebesgue measure on $\R^n$.  

\begin{proposition}[Density of epi-Lipschitzian sets]
Consider a set $Q\subset\R^n$ that is epi-Lipschitzian at a point $\bar{x}$. Then we have 
$$\mbox{\rm Dens}^{-}(Q,\bar{x}) > 0.$$
\end{proposition}
\begin{proof}
Since $Q$ is epi-Lipschitzian at $\bar{x}$, we may assume that $Q$ is an epigraph of a Lipschitz continuous function $f\colon\R^{n-1}\to\R$, with $\bar{x}=(0,0)$ (See for example \cite[Section 4]{Clarke_roc}). We deduce $f(x)\leq \kappa |x|$, for some constant $\kappa >0$ and for all points in $x\in\R^{n-1}$. Consequently, $\epi f$ contains a convex cone with nonempty interior, a set that has strictly positive Lebesgue density. The result follows.
\end{proof}

The sampling scheme outlined in this section is effective whenever gradients are cheap to compute and the normal cone to the domain at the point of interest is known in advance. Now suppose that the normal cone is unknown to us, but nevertheless we can test whether a point is in the domain and we can project onto the domain easily. Then a slight modification of the sampling scheme can still be effective at approximating the Clarke subdifferential. 
Namely during our sampling scheme, rather than discarding points lying outside of the domain, we may use these points to approximate the normal cone by projecting onto the domain. See \cite{har_lew} for more details. Using this information along with the sampled gradients, we may still hope to approximate the whole Clarke subdifferential effectively. We, however, do not pursue this further in our work.

%

\section{Local Dimension of semi-algebraic subdifferential graphs.}\label{sec:loc_dim}
In this section, we apply our methods to study the size of subdifferential graphs of semi-algebraic functions. We first establish notation and record some preliminary facts from semi-algebraic geometry.
\subsection{Semi-algebraic Geometry.}
A {\em semi-algebraic} set $S\subset\R^n$ is a finite union of sets of the form $$\{x\in \R^n: P_1(x)=0,\ldots,P_k(x)=0, Q_1(x)<0,\ldots, Q_l(x)<0\},$$ where $P_1,\ldots,P_k$ and $Q_1,\ldots,Q_l$ are polynomials in $n$ variables. In other words, $S$ is a union of finitely many sets, each defined by finitely many polynomial equalities and inequalities. A map $F\colon\R^n\rightrightarrows\R^m$ is said to be {\em semi-algebraic} if $\mbox{\rm gph}\, F\subset\R^{n+m}$ is a semi-algebraic set. Semi-algebraic sets enjoy many nice structural properties. Unless otherwise stated, we follow the notation of \cite{Coste-semi} and \cite{DM}. 

A fundamental fact about semi-algebraic sets is provided by the Tarski-Seidenberg Theorem \cite[Theorem 2.3]{Coste-semi}. Roughly speaking, it states that a linear projection of a semi-algebraic set remains semi-algebraic. From this result, it follows that a great many constructions preserve semi-algebraicity. In particular, for a semi-algebraic function $f\colon\R^n\to\overline{\R}$, the set-valued mappings $\partial_P f$, $\hat{\partial} f$, $\partial f$, and $\partial_c f$ are semi-algebraic. See for example \cite[Proposition 3.1]{tame_opt}. 

\begin{definition}[Compatibility]
{\rm Given finite collections $\{B_i\}$ and $\{C_j\}$ of subsets of $\R^n$, we say that $\{B_i\}$ is {\em compatible} with $\{C_j\}$ if for all $B_i$ and $C_j$, either $B_i\cap C_j=\emptyset$ or $B_i\subset C_j$.}
\end{definition}

\begin{definition}[Stratifications]\label{defn:whit}
{\rm Consider a set $Q$ in $\R^n$. A {\em stratification} of $Q$ is a finite partition of $Q$ into disjoint, connected, manifolds $M_i$ (called strata) with the property that for each index $i$, the intersection of the closure of $M_i$ with $Q$ is the union of some $M_j$'s.} 
\end{definition}

Remarkably, semi-algebraic sets always admit stratifications. Indeed, the following stronger result holds.
\begin{proposition}\label{prop:smooth}\cite[Theorem 4.8]{DM}
Consider a semi-algebraic function $f\colon\R^n\to\overline{\R}$. Then there exists a stratification $\mathcal{A}$ of $\dom f$ so that $f$ is smooth on each stratum $M_i\in \mathcal{A}$. Furthermore, if $\mathcal{B}$ is some other stratification of $\dom f$, then we can ensure that $\mathcal{A}$ is compatible with $\mathcal{B}$.
\end{proposition}


\begin{definition}[Dimension]
{\rm Let $Q\subset\R^n$ be a nonempty semi-algebraic set. Then the {\em dimension} of $Q$, denoted $\dim Q$, is the maximal dimension of a stratum in any stratification of $Q$. We adopt the convention that $\dim \emptyset=-\infty$.}  
\end{definition}

It can be easily shown that the dimension does not depend on the particular stratification. Observe that the dimension of a semi-algebraic set only depends on the maximal dimensional manifold in a stratification. Hence, dimension is a crude measure of the size of the semi-algebraic set. This motivates a localized notion of dimension.

\begin{definition}[Local dimension]
{\rm Consider a semi-algebraic set $Q\subset \R^n$ and a point $\bar{x}\in Q$. We let the {\em local dimension} of $Q$ at $\bar{x}$ be $$\dim_Q(\bar{x}):=\inf_{\epsilon>0}\dim (Q\cap B_{\epsilon}(\bar{x})).$$ It is not difficult to see that there exists a real number $\bar{\epsilon}>0$ such that for every real number $0<\epsilon<\bar{\epsilon}$, we have $\dim_Q(\bar{x})=\dim (Q\cap B_{\epsilon}(\bar{x}))$.
}
\end{definition}

There is a straightforward connection between local dimension and dimension of strata. This is the content of the following proposition.
\begin{proposition}\cite[Proposition 3.4]{small}\label{prop:id}
Consider a semi-algebraic set $Q\subset\R^n$ and a point $\bar{x}\in Q$. Let $\{M_i\}$ be any stratification of $Q$. Then we have the identity $$\dim_Q(\bar{x})=\max_i\{\dim M_i: \bar{x}\in\cl M_i\}.$$
\end{proposition}

\begin{definition}[Maximal strata]
{\rm Given a stratification $\{M_i\}$ of a semi-algebraic set $Q\subset\R^n$, a stratum $M$ is {\em maximal} if it is not contained in the closure of any other stratum.}
\end{definition}

Using the defining property of a stratification, we can equivalently say that given a stratification $\{M_i\}$ of a semi-algebraic set $Q\subset\R^n$, a stratum $M_i$ is maximal if and only if it is disjoint from the closure of any other stratum.

\begin{proposition}\cite[Proposition 3.7]{small}\label{prop:loc_max}
Consider a stratification $\{M_i\}$ of a semi-algebraic set $Q\subset\R^n$. Then given any point $\bar{x}\in Q$, there exists a maximal stratum $M$ satisfying $\bar{x}\in \cl M$ and $\dim M=\dim_Q(\bar{x})$. 
\end{proposition}

Semi-algebraic methods have recently found great uses in set-valued analysis. See for example \cite{Pang, dim, tame_opt, ioffe_tame, ioffe_strat}. A fact that will be particularly useful for us is that semi-algebraic set-valued mappings are ``generically'' inner-semicontinuous.
\begin{proposition}\label{prop:cont}\cite[Proposition 2.28, 2.30]{dim}
Consider a semi-algebraic, set-valued mapping $G\colon\R^n\rightrightarrows\R^m$. Then there exists a stratification of $\dom G$ into finitely many semi-algebraic manifolds $\{M_i\}$ such that on each stratum $M_i$, the mapping $G$ is inner-semicontinuous and the dimension of the images $F(x)$ is constant. If in addition $F$ is closed-valued, then we can ensure that the restriction $G\big|_{M_i}$ is also outer-semicontinuous for each index $i$.
\end{proposition}
For a more refined result along the lines of Proposition~\ref{prop:cont}, see \cite[Theorem 28]{Pang}. 
The following result is standard.
\begin{proposition} \cite[Theorem 3.18]{Coste-min} \label{prop:const_gen}
Consider a semi-algebraic, set-valued mapping $F\colon\R^n\rightrightarrows\R^m$. Suppose there exists an integer $k$ such that $F(x)$ is $k$-dimensional for each point $x\in \dom F$. Then the equality, $$\dim\gph F= \dim \dom F +k,$$ holds.   
\end{proposition}

We will need a version of Proposition~\ref{prop:const_gen} that pertains to local dimension.
\begin{proposition}\label{prop:loc_dim}
Consider a semialgebraic mapping $F\colon\R^n\rightrightarrows\R^m$ that is inner-semicontinuous on its domain. Suppose that there exist constants $k$ and $l$ such that for each pair $(x,v)\in\gph F$, we have
$$\dim_{\sdom F}x=k, ~\dim_{F(x)}v=l.$$ 
Then $\gph F$ has local dimension $k+l$ around every pair $(x,v)\in\gph F$.
\end{proposition}
\begin{proof} Let $\pi\colon\R^n\times\R^m\to\R^n$ be the canonical projection onto $\R^n$. Consider any stratification $\mathcal{A}$ of $\gph F$ and let $M\in \mathcal{A}$ be any maximal stratum. Clearly
$$\dim M\leq \dim \gph F\leq k+l.$$ 
Consider an arbitrary point $x\in \pi(M)$. Since $M$ is maximal, the set $M\cap (\{x\}\times\R^m)$ is open relative to $\gph F\cap (\{x\}\times\R^m)$. Furthermore, since $\dim_{F(x)}v=l$ for each vector $v\in F(x)$, it easily follows that $\dim M\cap (\{x\}\times\R^m)=l$.

We now claim that $\dim \pi(M)= k$. Indeed suppose this is not the case, that is the strict inequality $\dim \pi(M)< k$ holds. Since $\dim_{\sdom F}x=k$, we deduce that there exists a sequence $x_i\to\bar{x}$, with $x_i\in \dom F$ and $x_i\notin \pi(M)$ for each index $i$. Since $F$ is inner-semicontinuous on $M$, we deduce 
$$M\cap (\{x\}\times\R^m)\subset \limsupp_{i\to\infty} ~\{x_i\}\times F(x_i),$$ which contradicts maximality of $M$. Consequently, using Proposition~\ref{prop:const_gen}, we deduce $\dim M=k+l$. Since $M$ was an arbitrary maximal stratum, the result follows by an application of Proposition~\ref{prop:loc_max}.
\end{proof}

\subsection{Local dimension of semi-algebraic subdifferential graphs.}
We begin with a definition.
\begin{definition}[Subjets]
{\rm For a function $f\colon\R^n\to\overline{\R}$, the {\em limiting subjet} is given by
$$[\partial f]:=\{(x,f(x),v):v\in \partial f(x)\}.$$ Subjets corresponding to the other subdifferentials are defined analogously.} 
\end{definition}

Much like $f$-attentive convergence, subjets are useful for keeping track of variational information in absence of continuity.
In this section, we build on the following theorem. This result and its consequences for generic semi-algebraic optimization problems are discussed extensively in \cite{dim}. 
\begin{theorem}\cite[Theorem 3.6]{dim}\label{thm:grd}
Let $f\colon\R^n\rightarrow\overline{\R}$ be a proper semi-algebraic function. Then the subjets $[\partial_P f]$, $[\hat{\partial} f]$, $[\partial f]$ and $[\partial_c f]$ have dimension exactly $n$.
\end{theorem}
An immediate question arises: Can the four subjets have local dimension smaller than $n$ at some of their points? In a recent paper \cite{small}, the authors showed that this indeed may easily happen for $[\partial_c f]$. Remarkably the authors showed that the subjets $[\partial_P f]$, $[\hat{\partial} f]$, and $[\partial f]$ of a lower-semicontinuous, semi-algebraic function $f\colon\R^n\to\overline{\R}$ do have uniform local dimension $n$. The significance of this result and the relation to Minty's theorem were also discussed.
In this section, we provide a much simplified proof of this rather striking fact (Theorem~\ref{thm:loc_dim}).
The main tool we use is the following accessibility lemma, which is a special case of Lemma~\ref{lem:access_set_main}. Since the proof is much simpler than that of Lemmma~\ref{lem:access_set_main}, we include the full argument below. 
\begin{lemma}[Accessibility]\label{lem:access_prox}
Consider a closed set $Q\subset\R^n$, a manifold $M\subset Q$, and a point $\bar{x}\in M$. Recall that the inclusion $N^{P}_Q(\bar{x})\subset N_M(\bar{x})$ holds. Suppose that a proximal normal vector $\bar{v}\in N^{P}_Q(\bar{x})$ lies in the boundary of $N^{P}_Q(\bar{x})$, relative to the linear space $N_M(\bar{x})$. Then there exist sequences $x_i\to\bar{x}$ and $v_i\to\bar{v}$, with $v_i\in N_Q^{P}(x_i)$, and so that all the points $x_i$ lie outside of $M$.
\end{lemma}
\begin{proof} There exists a real number $\lambda >0$ so that $\bar{x}+\lambda\bar{v}$ lies in the prox-normal neighborhood $W$ of $M$ at $\bar{x}$ and such that the equality $P_Q(\bar{x}+\lambda\bar{v})=\bar{x}$ holds. Consider any sequence $v_i\in\R^n$ satisfying 
$$v_i\to\bar{v}, v_i\in N_M(\bar{x}), v_i\notin N_Q^{P}(\bar{x}).$$
Choose arbitrary points $x_i\in P_Q(\bar{x}+\lambda v_i)$. We have 
$$(\bar{x}- x_i)+\lambda v_i\in N_Q^{P}(x_i).$$
We deduce $x_i\neq \bar{x}$. Clearly, the sequence $x_i$ converges to $\bar{x}$. We claim $x_i\notin M$ for all sufficiently large indices $i$. Indeed, if it were otherwise, then for large $i$, the points $\bar{x}+\lambda v_i$ would lie in $W$ and we would have $x_i\in P_M(\bar{x}+\lambda v_i)=\bar{x}$, which is a contradiction. Thus we have obtained a sequence $(x_i,\frac{1}{\lambda}(\bar{x}- x_i)+v_i)\in \gph N_Q^{P}$, with $x_i\notin M$, and satisfying $(x_i,\frac{1}{\lambda}(\bar{x}- x_i)+v_i))\to(\bar{x},\bar{v})$.
\end{proof}

The following is now immediate.
\begin{corollary}\label{cor:access_func}
Consider a lower semicontinuous function $f\colon\R^n\to\overline{\R}$, a manifold $M\subset\R^n$, and a point $\bar{x}\in M$. Suppose that $f$ is smooth on $M$ and the strict inequality $\dim \partial_P f(\bar{x})<\dim N_M(\bar{x})$ holds. Then for every vector $\bar{v}\in \partial_{P} f(\bar{x})$, there exist sequences $(x_i,f(x_i),v_i)\to(\bar{x},f(\bar{x}),\bar{v})$, with $v_i\in \partial_P f(x_i)$, and so that all the points $x_i$ lie outside of $M$.
\end{corollary}
\begin{proof} From the strict inequality, one can easily see that the normal cone $N_{\sepi f}^P(\bar{x},f(\bar{x}))$ has empty interior relative to the normal space $N_{\sgph}(\bar{x},f(\bar{x}))$. An application of Lemma~\ref{lem:access_prox} completes the proof.
\end{proof}

We can now prove the main result of this section.
\begin{theorem}\label{thm:loc_dim}
Consider a lower-semicontinuous, semi-algebraic function $f\colon\R^n\to\overline{\R}$. Then the subjets $[\partial_P f]$, $[\hat{\partial} f]$, and $[\partial f]$ have constant local dimension $n$ around each of their points. 
\end{theorem}
\begin{proof} We first prove the theorem for the subjet $[\partial_P f]$. 
Consider the semi-algebraic set-valued mapping 
$$F(x):=\{f(x)\}\times \partial_P f(x),$$ whose graph is precisely $[\partial_P f]$. By Propositions~\ref{prop:smooth} and~\ref{prop:cont}, we may stratify the domain of $F$ into finitely many semi-algebraic manifolds $\{M_i\}$, so that on each stratum $M_i$, the mapping $F$ is inner-semicontinuous, the images $F(x)$ have constant dimension, and $f$ is smooth.
Consider a triple $(x,f(x),v)\in[\partial_P f]$. We prove the theorem by induction on the dimension of the strata $M$ in which the point $x$ lies. Clearly the result holds for the strata of dimension $n$, if there are any. As an inductive hypothesis, assume that the theorem holds for all points $(x,f(x),v)\in[\partial_P f]$ with $x$ lying in strata of dimension at least $k$, for some integer $k\geq 1$.

Now consider a stratum $M$ of dimension $k-1$ and a point $x\in M$. If $\dim F(x)=n-\dim M$, then recalling that $F$ is inner-semicontinuous on $M$ and applying Proposition~\ref{prop:loc_dim}, we see that the set $\gph F\Big|_M$ has local dimension $n$ around $(x,f(x),v)$ for any $v\in\partial_P f(x)$. The result follows in this case.

Now suppose $\dim F(x)< n-\dim M$. Then, by Corollary~\ref{cor:access_func}, for such a vector $v$, there exists a sequence $(x_i,f(x_i),v_i)\to (x,f(x),v)$  satisfying $(x_i,f(x_i),v_i)\in[\partial_P f]$ and $x_i\notin M$ for each index $i$. Restricting to a subsequence, we may assume that all the points $x_i$ lie in a stratum $K$ satisfying $\dim K\geq k$. By the inductive hypothesis, we deduce 
$$\dim_{[\partial_P f]} (x,f(x),v)\geq \lp_{i\to\infty} \dim_{[\partial_P f]} (x_i,f(x_i),v_i)=n.$$ This completes the proof of the inductive step and of the theorem for the subjet $[\partial_P f]$.

Now observe that $[\partial_P f]$ is dense in $[\hat{\partial} f]$ and in $[\partial f]$. It follows that $[\hat{\partial} f]$ and $[\partial f]$ also have local dimension $n$ around each of their points.
\end{proof}

Surprisingly Theorem~\ref{thm:loc_dim} may fail in the Clarke case, even for Lipschitz continuous functions.
\begin{example}
{\rm
Consider the function $f\colon\R^3\to\R,$ defined by
\begin{displaymath}
   f(x,y,z) = \left\{
     \begin{array}{ll}
       \min\{x,y,z^2\} &\mbox{\rm if }(x,y,z) \in \R_{+}^3\\
       \min\{-x,-y,z^2\} &\mbox{\rm if } (x,y,z) \in \R_{-}^3\\
       0 & \mbox{\rm{otherwise}.}       
     \end{array}
   \right.
\end{displaymath}
Let $\bar{x}\in\R^n$ be the origin and let $\Gamma:=\conv\{(1,0,0),(0,1,0),(0,0,0)\}$.
One can check that the local dimension of $\gph \partial_c f$ at $(\bar{x},\bar{v})$ is two for any vector $\bar{v}\in \big(\conv (\Gamma\cup -\Gamma)\big)\setminus (\Gamma\cup -\Gamma)$. For more details see \cite[Example 3.11]{small}.} 
\end{example}

\small

\bibliographystyle{plain}
\small
\parsep 0pt
\bibliography{dim_graph}

\begin{thebibliography}{10}

\bibitem{cone_mon}
J.M. Borwein, J.V. Burke, and A.S. Lewis.
\newblock Differentiability of cone-monotone functions on separable {B}anach
  space.
\newblock {\em Proc. Amer. Math. Soc.}, 132(4):1067--1076 (electronic), 2004.

\bibitem{Borwein-Zhu}
J.M. Borwein and Q.J. Zhu.
\newblock {\em Techniques of Variational Analysis}.
\newblock Springer Verlag, New York, 2005.

\bibitem{BLO}
J.V. Burke, A.S. Lewis, and M.L. Overton.
\newblock Approximating subdifferentials by random sampling of gradients.
\newblock {\em Mathematics of Operations Research}, 27(3):567--584, 2002.

\bibitem{alg}
J.V. Burke, A.S. Lewis, and M.L. Overton.
\newblock A robust gradient sampling algorithm for nonsmooth, nonconvex
  optimization.
\newblock {\em SIAM J. on Optimization}, 15:751--779, March 2005.

\bibitem{origin}
F.H. Clarke.
\newblock Generalized gradients and applications.
\newblock {\em Transactions of the American Mathematical Society},
  205:247--262, Apr. 1975.

\bibitem{CLSW}
F.H. Clarke, Yu. Ledyaev, R.I. Stern, and P.R. Wolenski.
\newblock {\em Nonsmooth Analysis and Control Theory}.
\newblock Texts in Math. 178, Springer, New York, 1998.

\bibitem{Coste-min}
M.~Coste.
\newblock {\em An Introduction to o-minimal Geometry}.
\newblock RAAG Notes, 81 pages, Institut de Recherche Math\'{e}matiques de
  Rennes, November 1999.

\bibitem{Coste-semi}
M.~Coste.
\newblock {\em An {I}ntroduction to {S}emialgebraic {G}eometry}.
\newblock RAAG Notes, 78 pages, Institut de Recherche Math\'{e}matiques de
  Rennes, October 2002.

\bibitem{Pang}
A.~Daniilidis and C.H.J. Pang.
\newblock Continuity and differentiability of set-valued maps revisited in the
  light of tame geometry.
\newblock {\em J. Lond. Math. Soc. (2)}, 83(3):637--658, 2011.

\bibitem{small}
D.~Drusvyatskiy, A.~D. Ioffe, and A.S. Lewis.
\newblock The dimension of semialgebraic subdifferential graphs.
\newblock {\em Nonlinear Anal.}, 75(3):1231--1245, 2012.

\bibitem{dim}
D.~Drusvyatskiy and A.S. Lewis.
\newblock Semi-algebraic functions have small subdifferentials.
\newblock {\em To appear in Mathematical Programming, Ser. B}, 2012.

\bibitem{CM}
A.~Marigonda G.~Colombo.
\newblock Singularities for a class of non-convex sets and functions, and
  viscosity solutions of some {Hamilton-Jacobi} equations.
\newblock {\em Journal of Convex Analysis}, 15(1):105 -- 129, 2008.

\bibitem{CMW}
P.~Wolenski G.~Colombo, A.~Marigonda.
\newblock The clarke generalized gradient for functions whose epigraph has
  positive reach.
\newblock {\em To appear in Mathematics of Operations Research}, 2012.

\bibitem{Hare}
W.L. Hare and A.S. Lewis.
\newblock Identifying active constraints via partial smoothness and
  prox-regularity.
\newblock {\em J. Convex Anal.}, 11(2):251--266, 2004.

\bibitem{har_lew}
W.L. Hare and A.S. Lewis.
\newblock Estimating tangent and normal cones without calculus.
\newblock {\em Math. Oper. Res.}, 30(4):785--799, 2005.

\bibitem{ioffe_tame}
A.D. Ioffe.
\newblock A {S}ard theorem for tame set-valued mappings.
\newblock {\em J. Math. Anal. Appl.}, 335(2):882--901, 2007.

\bibitem{ioffe_strat}
A.D. Ioffe.
\newblock Critical values of set-valued maps with stratifiable graphs.
  {E}xtensions of {S}ard and {S}male-{S}ard theorems.
\newblock {\em Proc. Amer. Math. Soc.}, 136(9):3111--3119, 2008.

\bibitem{tame_opt}
A.D. Ioffe.
\newblock An invitation to tame optimization.
\newblock {\em SIAM Journal on Optimization}, 19(4):1894--1917, 2009.

\bibitem{Lee}
J.M. Lee.
\newblock {\em Introduction to Smooth Manifolds}.
\newblock Springer, New York, 2003.

\bibitem{minty}
G.J. Minty.
\newblock Monotone (nonlinear) operators in {H}ilbert space.
\newblock {\em Duke Mathematical Journal}, 29:341--346, 1962.

\bibitem{Mord_1}
B.S. Mordukhovich.
\newblock {\em Variational Analysis and Generalized Differentiation I: Basic
  Theory}.
\newblock Grundlehren der mathematischen Wissenschaften, Vol 330, Springer,
  Berlin, 2006.

\bibitem{Mord_2}
B.S. Mordukhovich.
\newblock {\em Variational Analysis and Generalized Differentiation II:
  Applications}.
\newblock Grundlehren der mathematischen Wissenschaften, Vol 331, Springer,
  Berlin, 2006.

\bibitem{amen}
R.A. Poliquin and R.T. Rockafellar.
\newblock Amenable functions in optimization.
\newblock In {\em Nonsmooth optimization: methods and applications ({E}rice,
  1991)}, pages 338--353. Gordon and Breach, Montreux, 1992.

\bibitem{prox_reg}
R.A. Poliquin and R.T. Rockafellar.
\newblock Prox-regular functions in variational analysis.
\newblock {\em Trans. Amer. Math. Soc.}, 348:1805--1838, 1996.

\bibitem{newt}
S.M. Robinson.
\newblock A point-of-attraction result for {N}ewton's method with point-based
  approximations.
\newblock {\em Optimization}, 60(1-2):89--99, 2011.

\bibitem{rob}
S.M. Robinson.
\newblock Equations on monotone graphs.
\newblock {\em Mathematical Programming}, pages 1--53, 2012.

\bibitem{rock}
R.T. Rockafellar.
\newblock {\em Convex Analysis}.
\newblock Princeton University Press, 1970.

\bibitem{Clarke_roc}
R.T. Rockafellar.
\newblock Clarke's tangent cones and the boundaries of closed sets in {{\bf
  R}}$^n$.
\newblock {\em Nonlinear Analysis: Theory, Methods, and Applications}, 3(1):145
  -- 154, 1979.

\bibitem{VA}
R.T. Rockafellar and R.J-B. Wets.
\newblock {\em Variational {A}nalysis}.
\newblock Grundlehren der mathematischen Wissenschaften, Vol 317, Springer,
  Berlin, 1998.

\bibitem{DM}
L.~van~den Dries and C.~Miller.
\newblock Geometric categories and o-minimal structures.
\newblock {\em Duke Mathematical Journal}, 84:497--540, 1996.

\end{thebibliography}

%
%
%
%
%
%
%

\end{document}